\documentclass[11pt]{article}
\usepackage{times,amsmath,amsfonts,graphicx,amsthm,algorithms,fullpage,enumerate}
\usepackage{natbib}
\usepackage[all]{xy}

\newtheorem{theorem}{Theorem}
\newtheorem{lemma}[theorem]{Lemma}
\newtheorem{corollary}[theorem]{Corollary}
\newtheorem{proposition}[theorem]{Proposition}

\theoremstyle{definition}
\newtheorem{definition}{Definition}

\newtheorem{example}{Example}

\newcommand{\oracle}{\ensuremath{{\omega}}}
\newcommand{\esterr}{\ensuremath{\text{\sc err}}}
\newcommand{\crisk}{\ensuremath{{R}}}
\newcommand{\risk}{\ensuremath{{r}}}
\newcommand{\var}{\ensuremath{{\text{var}}}}
\newcommand{\cov}{\ensuremath{{\text{cov}}}}
\newcommand{\transpose}{\ensuremath{{\text{\sc T}}}}
\newcommand{\Gsw}{\ensuremath{G_{\text{\rm sw}}}}
\newcommand{\Gmmvar}{\ensuremath{G_{\text{\rm mm-var}}}}
\newcommand{\Gmmfix}{\ensuremath{G_{\text{\rm mm-fix}}}}
\newcommand{\dif}{\,\text{d}}
\newcommand{\gmm}{\ensuremath{g_{\text{\rm mm}}}}

\newcommand{\model}{\ensuremath{\mathcal{M}}}
\newcommand{\modelclass}{\ensuremath{\mathcal{M}^*}}

\newcommand{\samplespace}{\ensuremath{\mathcal{X}}}

\newcommand{\mltheta}{\ensuremath{\hat{\theta}}}
\newcommand{\Union}{\ensuremath{\bigcup}}
\newcommand{\Intersection}{\ensuremath{\bigcap}}
\newcommand{\union}{\ensuremath{\cup}}
\newcommand{\intersection}{\ensuremath{\cap}}

\newcommand{\naturals}{\ensuremath{\mathbb{N}}}

\newcommand{\posints}{\ensuremath{\mathbb{Z}^+}}
\newcommand{\reals}{\ensuremath{\mathbb{R}}}
\newcommand{\basisfamily}{\ensuremath{\mathcal{Q}}}

\newcommand{\ceil}[1]{\lceil#1\rceil}

\newcommand{\ind}{\ensuremath{{\bf 1}}}             

\newcommand{\m}[1]{\ensuremath{\bar{#1}}}           
\newcommand{\mbayes}[1]{\ensuremath{\bar{#1}^{\text{B}}}}         
\newcommand{\Pswitch}{\ensuremath{P_\text{\rm sw}}}
\newcommand{\Pmm}{\ensuremath{P_\text{\rm mm-fix}}}
\newcommand{\Pcesaro}{\ensuremath{P_\text{\rm Ces\`aro}}}
\newcommand{\Pcesaroswitch}{\ensuremath{P_\text{\rm Ces\`aro-sw}}}
\newcommand{\pswitch}{\ensuremath{p_\text{\rm sw}}}
\newcommand{\Pbayes}{\ensuremath{\bar{P}}} 
\newcommand{\pbayes}{\ensuremath{\bar{p}}} 
\newcommand{\Pbma}{\ensuremath{P_\text{\rm bma}}}
\newcommand{\pbma}{\ensuremath{p_\text{\rm bma}}}

\newcommand{\switchpars}{\ensuremath{\mathbb{S}}}
\newcommand{\switchpar}{\ensuremath{\mathbf{s}}}
\newcommand{\switchprior}{\ensuremath{\pi}}
\newcommand{\pim}{\ensuremath{\switchprior_\textsc{\tiny{m}}}}
\newcommand{\pit}{\ensuremath{\switchprior_\textsc{\tiny{t}}}}
\newcommand{\pik}{\ensuremath{\switchprior_\textsc{\tiny{k}}}}
\renewcommand{\vec}[1]{\ensuremath{{\bf #1}}}

\newcommand{\commentout}[1]{}

\DeclareMathOperator{\argmin}{arg\,min}
\DeclareMathOperator{\argmax}{arg\,max}

\title{Catching Up Faster by Switching Sooner:\footnote{A preliminary
    version of a part of this paper appeared as \citep{ErvenGR07}.} \\
  {\em A Prequential Solution to the AIC-BIC Dilemma}}

\author{%
{Tim van Erven \qquad Peter Gr\"unwald \qquad Steven de Rooij}\\
Centrum voor Wiskunde en Informatica (CWI)\\
Kruislaan 413, P.O. Box 94079 \\
1090 GB Amsterdam, The Netherlands \\
\texttt{\{Tim.van.Erven,Peter.Grunwald,Steven.de.Rooij\}@cwi.nl}}

\begin{document}
\bibliographystyle{abbrvnat}

\maketitle

\begin{abstract}
  Bayesian model averaging, model selection and its approximations 
  such as BIC are generally
  statistically consistent, but sometimes achieve slower rates of
  convergence than other methods such as AIC and leave-one-out
  cross-validation. On the other hand, these other methods can be
  inconsistent. We identify the \emph{catch-up phenomenon} as a novel
  explanation for the slow convergence of Bayesian methods. Based on
  this analysis we define the switch distribution, a modification of
  the Bayesian marginal distribution. We show that, under broad conditions,
  model selection and prediction based on the switch distribution is
  both consistent and achieves optimal convergence rates, thereby
  resolving the AIC-BIC dilemma. The method is practical; we give an
  efficient implementation. The switch distribution has a data
  compression interpretation, and can thus be viewed as a
  ``prequential'' or MDL method; yet it is different from the MDL
  methods that are usually considered in the literature. We compare
  the switch distribution to Bayes factor model selection and
  leave-one-out cross-validation.
\end{abstract}

\section{Introduction: The Catch-Up Phenomenon}
\label{sec:introduction}
We consider inference based on a countable set of models
(sets of probability distributions), focusing on two tasks:
model selection and model averaging. In model selection tasks, the
goal is to select the model that best explains the given data. In
model averaging, the goal is to find the weighted combination of
models that leads to the best prediction of future data from
the same source.

An attractive property of some criteria for model selection is that
they are consistent under weak conditions, i.e.\ if the
true distribution $P^*$ is in one of the models, then the
$P^*$-probability that this model is selected goes to one as the
sample size increases.  BIC \citep{schwarz1978}, Bayes factor model
selection \citep{kass1995}, Minimum Description Length (MDL) model
selection \citep{barron1998b} and prequential model validation
\citep{dawid1984} are examples of widely used model selection criteria
that are usually consistent. However, other model selection criteria
such as AIC \citep{akaike1974} and leave-one-out cross-validation
(LOO) \citep{stone1977}, while often inconsistent, do typically yield
better predictions. This is especially the case in nonparametric
settings of the following type: $P^*$ can be arbitrarily well-approximated by a
sequence of distributions in the (parametric) models under
consideration, but is not itself contained in any of these. In many
such cases, the predictive distribution converges to the true
distribution at the optimal rate for AIC and LOO
\citep{Shibata83,Li87}, whereas in general MDL, BIC, the Bayes factor
method and prequential validation only achieve the optimal rate to
within an $O(\log n)$ factor
\citep{rissanen1992,FosterG94,Yang99,grunwald2007}. In this paper we reconcile
these seemingly conflicting approaches \citep{Yang05a} by improving
the rate of convergence achieved in Bayesian model selection without
losing its consistency properties. First we provide an example to show
why Bayes sometimes converges too slowly.

\subsection{The Catch-Up Phenomenon}
Given priors on parametric  models $\model_1$, $\model_2$, $\ldots$ and parameters
therein, Bayesian inference associates each model $\model_k$ with the
marginal distribution $\m{p}_k$, given by
$$
  \pbayes_k(x^n) = \int_{\theta \in \Theta_k} 
    p_{\theta}(x^n) w(\theta)\dif\theta.
$$ 
obtained by averaging over the parameters according to the prior. In
Bayes factor model selection the preferred model is the one with maximum a posteriori
probability. By Bayes' rule
this is $\argmax_k \m{p}_k(x^n)w(k)$, where $w(k)$ denotes
the prior probability of $\model_k$.
We can further average over model indices, a process called
Bayesian Model Averaging (BMA). The resulting distribution $\pbma(x^n) =
\sum_k \m{p}_k(x^n) w(k)$ can be used for prediction.
%
%
In a sequential setting, the probability of a data sequence
$x^n:=x_1,\ldots,x_n$ under a distribution $p$ typically decreases
exponentially fast in $n$. It is therefore common to consider $-\log
p(x^n)$, which we call the \emph{code length} of $x^n$ achieved by
$p$. We take all logarithms to base $2$, allowing us to measure
code length in \emph{bits}. The name code length refers to the
correspondence between code length functions and probability
distributions based on the Kraft inequality, but one may also think of
the code length as the accumulated log loss that is incurred if we
sequentially predict the $x_i$ by conditioning on the past, i.e.\
using $p(\cdot | x^{i-1})$
\citep{barron1998b,grunwald2007,dawid1984,rissanen1984}. For BMA, we
have 
$$-\log \pbma(x^n) = 
- \log \prod_{i=1}^n \pbma(x_i \mid x^{i-1})
= \sum_{i=1}^n \left[ -\log \pbma(x_i \mid
x^{i-1}) \right].
$$\label{eq:bayesmargintro}Here the $i$th term represents the
loss incurred when predicting $x_i$ given $x^{i-1}$ using $\pbma(\cdot
| x^{i-1})$, which turns out to be equal to the posterior average:
$\pbma(x_i | x^{i-1}) = \sum_k \m{p}_k(x_i | x^{i-1}) w(k| x^{i-1})$.

Prediction using $\pbma$ has the advantage that the code length it
achieves on $x^n$ is close to the code length of $\m{p}_{\hat{k}}$, where
$\hat{k}$ is the best of the marginals $\m{p}_1, \m{p}_2, \ldots$, i.e.
$\hat{k}$ achieves $\min_{k} - \log \m{p}_k(x^n)$.  More precisely, given
a prior $w$ on model indices, the difference between $-\log
\pbma(x^n)=-\log(\sum_k \m{p}_k(x^n)w(k))$ and $-\log \m{p}_{\hat{k}}(x^n)$
must be in the range $[0,-\log w(\hat{k})]$, whatever data $x^n$ are
observed. Thus, using BMA for prediction is sensible if we are
satisfied with doing essentially as well as the best model under
consideration. However, it is often possible to combine
$\m{p}_1,\m{p}_2,\ldots$ into a distribution that achieves smaller code length
than $\m{p}_{\hat{k}}$! This is possible if the index $\hat{k}$ of the
best distribution \emph{changes with the sample size in a predictable
  way}. This is common in model selection, for example with nested
models, say $\model_1\subset\model_2$. In this case $\m{p}_1$ typically
predicts better at small sample sizes (roughly, because $\model_2$ has
more parameters that need to be learned than $\model_1$), while $\m{p}_2$
predicts better eventually.  Figure~\ref{fig:markovexample}
illustrates this phenomenon. 
\begin{figure}[t]
\begin{center}
  \includegraphics[width=0.54\columnwidth]{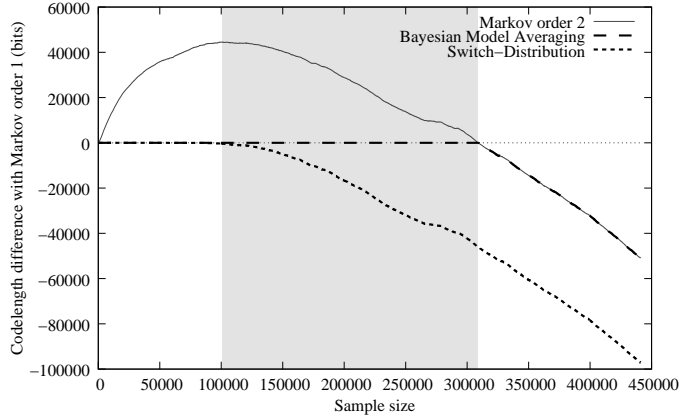}
  \caption{The Catch-up Phenomenon}
  \label{fig:markovexample}
\end{center}
\end{figure}
It shows the accumulated code length
difference $ - \log \m{p}_2(x^n) - ( - \log \m{p}_1(x^n))$ on ``The Picture of
Dorian Gray'' by Oscar Wilde, where $\m{p}_1$ and $\m{p}_2$ are the Bayesian
marginal distributions for the first-order and second-order Markov
chains, respectively, and each character in the book is an outcome.
We used uniform
(Dirichlet$(1,1,\ldots, 1)$) priors on the
model parameters (i.e., the ``transition probabilities'')
, but the same phenomenon occurs with other common priors , such as Jeffreys''. Clearly $\m{p}_1$ is better for about the first $100\,000$
outcomes, gaining a head start of approximately $40\,000$ bits.
Ideally we should predict the initial $100\,000$ outcomes using $\m{p}_1$
and the rest using $\m{p}_2$. However, $\pbma$ only starts to behave like
$\m{p}_2$ when it \emph{catches up} with $\m{p}_1$ at a sample size of about
$310\,000$, when the code length of $\m{p}_2$ drops below that of $\m{p}_1$.
Thus, in the shaded area $\pbma$ behaves like $\m{p}_1$ while $\m{p}_2$ is
making better predictions of those outcomes: since at $n=100\,000$,
$\m{p}_2$ is $40\,000$ bits behind, and at $n= 310\,000$, it has caught
up, in between it must have outperformed $\m{p}_1$ by $40\,000$ bits!

Note that the example models $\model_1$ and $\model_2$ are very crude;
for this particular application much better models are available. Thus
$\model_1$ and $\model_2$ serve as a simple illustration only (see the
discussion in Section~\ref{sec:unbelievabletruth}).  However, our
theorems, as well as experiments with nonparametric density estimation
on which we will report elsewhere, indicate that the same phenomenon
also occurs with more realistic models. In fact, the general pattern that first
one model is better and then another occurs widely, both on real-world
data and in theoretical settings. We argue that failure to take this
effect into account leads to the suboptimal rate of convergence
achieved by Bayes factor model selection and related methods.  We have
developed an alternative method to combine distributions $\m{p}_1$ and
$\m{p}_2$ into a single distribution $\pswitch$, which we call the
\emph{switch distribution}, defined in Section~\ref{sec:switch
  distribution}.  Figure~\ref{fig:markovexample} shows that $\pswitch$
behaves like $\m{p}_1$ initially, but in contrast to $\pbma$ it starts
to mimic $\m{p}_2$ \emph{almost immediately} after $\m{p}_2$ starts
making better predictions; it essentially does this \emph{no matter
  what sequence $x^n$ is actually observed}.  $\pswitch$ differs from
$\pbma$ in that it is based on a prior distribution on \emph{sequences
  of models} rather than simply a prior distribution on models. This
allows us to avoid the implicit assumption that there is one model
which is best at all sample sizes. After conditioning on past
observations, the posterior we obtain gives a better indication of
which model performs best \emph{at the current sample size}, thereby
achieving a faster rate of convergence. Indeed, the switch
distribution is very closely related to earlier algorithms for
\emph{tracking the best expert} developed in the universal prediction
literature; see also Section~\ref{sec:relevance}
\citep{HerbsterWarmuth1998,Vovk1999,volfwillems1998,MonteleoniJ04};
however, the applications we have in mind and the theorems we prove
are completely different.

\subsection{Organization}
The remainder of the paper is organized as follows (for the reader's
convenience, we have attached a table of contents at the end of the paper). In
Section~\ref{sec:switch distribution} we introduce our basic concepts
and notation, and we then define the switch distribution. While in the
example above, we switched between just two models, the general
definition allows switching between elements of any finite or
countably infinite set of models.  In Section~\ref{sec:consistency} we
show that model selection based on the switch distribution is
consistent (Theorem~\ref{thm:consistencyagain}). Then in
Section~\ref{sec:nonparametrica} we show that the switch distribution
achieves a rate of convergence that is never significantly worse than
that of Bayesian model averaging, and we show that, in contrast to
Bayesian model averaging, the switch distribution achieves the
\emph{worst-case optimal} rate of convergence when it is applied to
histogram density estimation. In Section~\ref{sec:nonparametricb} we
develop a number of tools that can be used to bound the rate of
convergence in Ces\`aro-mean in more general parametric and
nonparametric settings, which include histogram density estimation as
a special case. In Section~\ref{sec:exponential} and
Section~\ref{sec:linear} we apply these tools to show that the switch
distribution achieves minimax convergence rates in density estimation
based on exponential families and in some nonparametric linear
regression problems. In Section~\ref{sec:computation} we give a
practical algorithm that computes the switch distribution.
Theorem~\ref{thm:algo} of that section shows that the run-time for $k$
predictors is $\Theta(n\cdot k)$ time. In Sections~\ref{sec:relevance}
and Section~\ref{sec:ketchup} we put our work in a broader context and
explain how our results fit into the existing literature.
Specifically, Section~\ref{sec:aicbic} explains how our result can be
reconciled with a seemingly contradictory recent result of \citet{Yang05a}, and
Section~\ref{sec:unbelievabletruth} describes a strange
implication of the catch-up phenomenon for Bayes factor model
selection. The proofs of all theorems are in Appendix~\ref{sec:proofs}
(except the central results of Section~\ref{sec:nonparametricb}, which
are proved in the main text).
\section{The switch distribution for Model Selection and Prediction}
\label{sec:switch distribution}

\subsection{Preliminaries}
\label{sec:preliminaries}

Suppose $X^\infty = (X_1$, $X_2$, $\ldots)$ is a sequence of random
variables that take values in sample space $\samplespace \subseteq
\reals^d$ for some $d \in \posints = \{1,2,\ldots\}$. For $n \in
\naturals = \{0,1,2,\ldots\}$, let $x^n = (x_1$, $\ldots$, $x_n)$ denote
the first $n$ outcomes of $X^\infty$, such that $x^n$ takes values in
the product space $\samplespace^n = \samplespace_1 \times \cdots \times
\samplespace_n$. (We let $x^0$ denote the empty sequence.) 
For $m > n$, we
write $X_{n+1}^m$ for $(X_{n+1}$, $\ldots$, $X_m)$, where $m = \infty$
is allowed. We omit the subscript when $n=0$, writing $X^m$ rather than
$X^m_1$.

Any distribution $P(X^\infty)$ may be defined in terms of a sequential
\emph{prediction strategy} $p$ that predicts the next outcome at any
time $n \in \naturals$. To be precise: Given the previous outcomes
$x^n$ at time $n$, a prediction strategy should issue a conditional
density $p(X_{n+1}|x^n)$ with corresponding distribution
$P(X_{n+1}|x^n)$ for the next outcome $X_{n+1}$. Such sequential
prediction strategies are sometimes called \emph{prequential
  forecasting systems} \citep{dawid1984}. An instance is given in
Example~\ref{ex:bern} below. Whenever the existence of a `true'
distribution $P^*$ is assumed --- in other words, $X^\infty$ are
distributed according $P^*$ ---, we may think of any prediction
strategy $p$ as a procedure for estimating $P^*$, and in such cases,
we will often refer to $p$ an \emph{estimator}. For simplicity, we
assume throughout that the density $p(X_{n+1}| x^n)$ is taken relative
to either the usual Lebesgue measure (if $\samplespace$ is continuous)
or the counting measure (if $\samplespace$ is countable). In the
latter case $p(X_{n+1}| x^n)$ is a probability mass function.  It is
natural to define the joint density $p(x^m|x^n) = p(x_{n+1}|x^n)
\cdots p(x_m|x^{m-1})$ and let $P(X_{n+1}^\infty|x^n)$ be the unique
distribution on $\samplespace^\infty$ such that, for all $m > n$,
$p(X_{n+1}^m|x^n)$ is the density of its marginal distribution for
$X_{n+1}^m$. To ensure that $P(X_{n+1}^\infty|x^n)$ is well-defined
even if $\samplespace$ is continuous, we will only allow prediction
strategies satisfying the natural requirement that for any $k \in
\posints$ and any fixed measurable event $A_{k+1} \subseteq
\samplespace_{k+1}$ the probability $P(A_{k+1} | x^k)$ is a measurable
function of $x^k$. This requirement holds automatically if
$\samplespace$ is countable.

\subsection{Model Selection and Prediction}

In \emph{model selection} the goal is to choose an explanation for
observed data $x^n$ from a potentially infinite list of candidate models
$\model_1$, $\model_2$, $\ldots$ We consider \emph{parametric models},
which we define as sets $\{p_\theta:\theta\in\Theta\}$ of prediction
strategies $p_\theta$ that are indexed by elements of $\Theta \subseteq
\reals^d$, for some smallest possible $d \in \naturals$, the number of
degrees of freedom. A model is more commonly viewed as a set of
distributions, but since distributions can be viewed as prediction
strategies as explained above, we may think of a model as a set of
prediction strategies as well. Examples of model selection are histogram
density estimation \citep{rissanen1992} ($d$ is the number of bins minus
1), regression based on a set of basis functions such as polynomials
($d$ is the number of coefficients of the polynomial), and the variable
selection problem in regression \citep{Shibata83,Li87,Yang99} ($d$ is the
number of variables). A \emph{model selection criterion} is a function
$\delta: \Union_{n=0}^\infty \samplespace^n \rightarrow \posints$ that, given any data
sequence $x^n \in \samplespace^n$ of arbitrary length $n$, selects the model $\model_k$ with
index $k=\delta(x^n)$.

With each model $\model_k$ we associate a single prediction strategy
$\m{p}_k$. The bar emphasizes that $\m{p}_k$ is a
meta-strategy based on the prediction strategies in $\model_k$. In many
approaches to model selection, for example AIC and LOO, $\m{p}_k$ is
defined using some parameter estimator $\mltheta_k$, which maps a
sequence $x^n$ of previous observations to an estimated parameter value
that represents a ``best guess'' of the true/best distribution in the
model. Prediction is then based on this estimator: $\m{p}_k(X_{n+1}\mid
x^n) =p_{\mltheta_k(x^n)}(X_{n+1}\mid x^n)$, which also defines a joint
density $\m{p}_k(x^n)=\m{p}_k(x_1)\cdots\m{p}_k(x_n|x^{n-1})$. The
Bayesian approach to model selection or model averaging goes the other
way around. It starts out with a prior $w$ on $\Theta_k$, and then
defines the Bayesian marginal density
\begin{equation}
  \label{eq:bayesint}
  \pbayes_k(x^n) = \int_{\theta \in \Theta_k} 
    p_{\theta}(x^n) w(\theta)\dif\theta.
\end{equation}
When $\pbayes_k(x^n)$ is non-zero this joint density induces a unique
conditional density $$\pbayes_k(X_{n+1} \mid x^n ) =
\frac{\pbayes_k(X_{n+1},x^n)}{\pbayes_k(x^n)},$$ 
which is equal to the mixture
of $p_\theta$ according to the posterior, $w(\theta|x^n)
= p_\theta(x^n)w(\theta)/\int p_\theta(x^n)w(\theta)\dif\theta$, based on
$x^n$. Thus the Bayesian approach also defines a prediction strategy
$\pbayes_k(X_{n+1}|x^n)$. 

Associating a prediction strategy $\m{p}_k$ with each model $\model_k$
is known as the \emph{prequential approach to statistics}
\citep{dawid1984} or \emph{predictive MDL} \citep{rissanen1984}.
Regardless of whether $\m{p}_k$ is based on parameter estimation or on
Bayesian predictions, we may usually think of it as a universal code
relative to $\model_k$ \citep{grunwald2007}.

\begin{example}
\label{ex:bern}
Suppose $\samplespace = \{0,1\}$. Then a prediction strategy $\m{p}$
may be based on the Bernoulli model $\model = \{ p_{\theta} \mid
\theta \in [0,1] \}$ that regards $X_1, X_2, \ldots$ as a sequence of
independent, identically distributed Bernoulli random variables with
$P_\theta(X_{n+1}=1) = \theta$. We may predict $X_{n+1}$ using the
maximum likelihood (ML) estimator based on the past, i.e.\ using
$\mltheta(x^n) = n^{-1} \sum_{i=1}^n x_i$. The prediction for $x_1$ is
then undefined.  If we use a smoothed ML estimator such as the Laplace
estimator, $\hat{\theta}'(x^n) = (n+2)^{-1} (\sum_{i=1}^n x_i + 1)$,
then all predictions are well-defined. It is well-known that the
predictor $\m{p}'$ defined by $\m{p}'(X_{n+1} \mid x^n) =
p_{\mltheta'(x^n)}(X_{n+1})$ equals the Bayesian predictive
distribution based on a uniform prior. Thus in this case a Bayesian
predictor and an estimation-based predictor coincide! 

In general, for a parametric model $\model_k$, we can define
$\m{p}_k(X_{n+1} \mid x^n) = p_{{\mltheta'}_k(x^n)}(X_{n+1})$ for some
smoothed ML estimator $\mltheta'_k$. The joint distribution with density
$\m{p}_k(x^n)$ will then resemble, but in general not be precisely equal
to, the Bayes marginal distribution with density $\pbayes_k(x^n)$ under
some prior on $\model_k$ \citep[Chapter 9]{grunwald2007}.
\end{example}
\subsection{The switch distribution}
\label{sec:switchdefinition}

Suppose $p_1$, $p_2$, $\ldots$ is a list of prediction strategies for
$X^\infty$. (Although here the list is infinitely long, the developments
below can with little modification be adjusted to the case where the
list is finite.) We first define a family
$\basisfamily=\{q_\switchpar:\switchpar\in\switchpars\}$ of combinator
prediction strategies that switch between the original prediction
strategies. Here the parameter space $\switchpars$ is defined as
\begin{equation}
  \switchpars = \{(t_1,k_1),\ldots,(t_m,k_m) \in (\naturals \times \posints)^m \mid m\in\posints,
        0 = t_1 < \ldots < t_m\}.
\end{equation}
The parameter $\switchpar \in \switchpars$ specifies the identities of
$m$ constituent prediction strategies and the sample sizes, called
\emph{switch-points}, at which to switch between them. For
$\switchpar=((t'_1,k'_1),\ldots,(t'_{m'},k'_{m'}))$, let
$t_i(\switchpar)=t'_i$, $k_i(\switchpar)=k'_i$ and $m(\switchpar)=m'$.
We omit the argument when the parameter $\switchpar$ is clear from
context; e.g.\ we write $t_3$ for $t_3(\switchpar)$. For each
$\switchpar \in \switchpars$ the corresponding $q_\switchpar \in
\basisfamily$ is defined as:
\begin{equation}
\label{eq:switchbasis}
  q_\switchpar(X_{n+1} | x^n) =
    \begin{cases}
      \hfill p_{k_1}(X_{n+1} | x^n) & \text{if $n < t_2$,} \\
      \hfill p_{k_2}(X_{n+1} | x^n) & \text{if $t_2 \leq n < t_3$,} \\
      \hfill \vdots \hfill& \hfill\vdots\hfill\\
      p_{k_{m-1}}(X_{n+1} | x^n) & \text{if $t_{m-1} \leq n < t_m$,} \\
      \hfill p_{k_m}(X_{n+1} | x^n) & \text{if $t_m \leq n$}.
    \end{cases}
\end{equation}
Switching to the same predictor multiple times (consecutively or not) is
allowed. The extra switch-point $t_1$ is included to simplify notation;
we always take $t_1 = 0$, so that $k_1$ represents the strategy that is
used in the beginning, before any actual switch takes place.

Given a list of prediction strategies $p_1$, $p_2$, $\ldots$, we define
the switch distribution as a Bayesian mixture of the elements of
$\basisfamily$ according to a prior $\switchprior$ on $\switchpars$:
\begin{definition}[switch distribution]
  \label{def:switchdist}
  Suppose $\switchprior$ is a probability mass function on
  $\switchpars$. Then the \emph{switch distribution} $\Pswitch$ with
  prior $\switchprior$ is the distribution for $(X^\infty,\switchpar)$
  that is defined by the density
  \begin{equation}
    \label{eq:switch}
    \pswitch(x^n,\switchpar) = q_\switchpar(x^n) \cdot \switchprior(\switchpar)
  \end{equation}
  for any $n \in \posints$, $x^n \in \samplespace^n$, and $s \in \switchpars$.
\end{definition}
Hence the marginal likelihood of the switch distribution has density
\begin{equation}
    \pswitch(x^n)
      = \sum_{\switchpar \in \switchpars}
        q_\switchpar(x^n) \cdot \switchprior(\switchpar).
\end{equation}
Although the switch distribution provides a general way to combine
prediction strategies (see Section~\ref{sec:predictionexpert}), in this paper
it will only be applied to combine prediction strategies $\pbayes_1$,
$\pbayes_2$, $\ldots$ that correspond to parametric models. In this
case we may define a corresponding model selection criterion
$\delta_\text{sw}$. To this end, let  
$K_{n+1}:\switchpars\rightarrow\posints$ be a random variable that
denotes the strategy/model that is used to predict $X_{n+1}$ given
past observations $x^n$. Formally, let $i_0$ be the unique $i$ such
that $t_i(\switchpar) \leq n$ and either $t_{i+1}(\switchpar) > n$
(i.e.\ the current sample size $n$ is between the $i$-th and $i+1$-st
switch-point), or 
$i=m(\switchpar)$ (i.e.\ the current sample size $n$ is beyond the last
switch point). Then 
$K_{n+1}(\switchpar)=k_{i_0}(\switchpar)$. Now note that by Bayes' theorem,
the prior $\pi$, together with the data $x^n$, induces a posterior
$\pi({\bf s} \mid x^n) \propto q_{\bf s}(x^n) \pi({\bf s})$ on
switching strategies ${\bf s}$. This posterior on switching strategies
further induces a posterior on the model $K_{n+1}$ that is used to
predict $X_{n+1}$. 
Algorithm~\ref{algo:switch}, given in Section~\ref{sec:computation},
efficiently computes the posterior distribution on $K_{n+1}$ given
$x^n$:
\begin{equation}
\label{eq:posteriork}
  \pi(K_{n+1} = k \mid x^n)
    = \frac{\sum_{\{\switchpar : K_{n+1}(\switchpar) = k\}}
         q_{\switchpar}(x^n)\pi(\switchpar)}
    {\pswitch(x^n)},
\end{equation}
which is defined whenever $\pswitch(x^n)$ is non-zero, and can be
efficiently computed using Algorithm~\ref{algo:switch} (see
Section~\ref{sec:computation}). We turn this posterior distribution into
the model selection criterion 
\begin{equation}
  \delta_\text{sw}(x^n) = \arg \max_k \pi(K_{n+1}=k \mid x^n),
\end{equation}
which selects the model with maximum posterior probability.

\section{Consistency}
\label{sec:consistency}

If one of the models, say with index $k^*$, is actually true, then it
is natural to ask whether $\delta_\text{sw}$ is \emph{consistent}, in
the sense that it asymptotically selects $k^*$ with probability $1$.
Theorem~\ref{thm:consistencyagain} below states that, if the prediction
strategies $\m{p}_k$ associated with the models are Bayesian predictive
distributions, then $\delta_{\text{sw}}$ is consistent under certain
conditions which are only slightly stronger than those required for
standard Bayes factor model selection consistency. It is followed by
Theorem~\ref{thm:consistencyb}, which extends the result to the
situation where the $\m{p}_k$ are not necessarily Bayesian.

Bayes factor model selection is consistent if for all $k, k' \neq k$,
$\Pbayes_k(X^\infty)$ and $\Pbayes_{k'}(X^\infty)$ are mutually
singular, that is, if there exists a measurable set $A \subseteq
\samplespace^\infty$ such that $\Pbayes_k(A)=1$ and $\Pbayes_{k'}(A) =
0$~\citep{barron1998b}. For example, this can usually be shown to hold
if (a) the models are nested and (b) for each $k$, $\Theta_k$ is a
subset of $\Theta_{k+1}$ of $w_{k+1}$-measure $0$. In most interesting
applications in which (a) holds, (b) also holds \citep{grunwald2007}.
For consistency of $\delta_\text{sw}$, we need to strengthen the
mutual singularity-condition to a ``conditional''
mutual singularity-condition: we require that, for all $k'
\neq k$ and all $n$, all $x^n \in \samplespace^n$, the distributions
$\Pbayes_k(X_{n+1}^\infty \mid x^n)$ and $\Pbayes_{k'}(X_{n+1}^\infty
\mid x^n)$ are mutually singular. For example, if $X_1, X_2, \ldots$
are independent and identically distributed (i.i.d.) according to each
$P_\theta$ in all models, but also if $\samplespace$ is countable and
$\m{p}_k(x_{n+1} \mid x_n) > 0$ for all $k$, all $x^{n+1} \in
\samplespace^{n+1}$, then this conditional mutual singularity is
automatically implied by ordinary mutual singularity of
$\Pbayes_k(X^\infty)$ and $\Pbayes_{k'}(X^\infty)$.

Let $E_\switchpar = \{\switchpar' \in \switchpars \mid m(\switchpar')
> m(\switchpar), (t_i(\switchpar'),k_i(\switchpar'))=(t_i(\switchpar),k_i(\switchpar)) \text{ for $i = 1,\ldots,m(\switchpar)$}\}$ denote the set of all
possible extensions of $\switchpar$ to more switch-points. Let $\pbayes_1$,
$\pbayes_2$, $\ldots$ be Bayesian prediction strategies with respective
parameter spaces $\Theta_1$, $\Theta_2$, $\ldots$ and priors $w_1$,
$w_2$, $\ldots$, and let $\pi$ be the prior of the corresponding
switch distribution.

\begin{theorem}[Consistency of the switch distribution]
\label{thm:consistencyagain}

Suppose $\pi$ is positive everywhere on $\{\switchpar \in \switchpars \mid
m(\switchpar)=1\}$ and such that for some positive constant $c$,
for every $\switchpar \in \switchpars$, $c \cdot \pi(\switchpar) \geq
\pi(E_\switchpar)$. Suppose further that $\Pbayes_k(X_{n+1}^\infty \mid
x^n)$ and $\Pbayes_{k'}(X_{n+1}^\infty \mid x^n)$ are mutually singular
for all $k, k' \in \posints$, $k \neq k'$, all $n$, all 
$x^n \in \samplespace^n$.
Then, for all $k^* \in \posints$, for all $\theta^* \in \Theta_{k^*}$
except for a subset of $\Theta_{k^*}$ of $w_{k^*}$-measure $0$, the
posterior distribution on $K_{n+1}$ satisfies
\begin{equation}
  \label{eqn:thmconsistency}
  \pi(K_{n+1} = k^* \mid X^n) \overset{n \to \infty}{\longrightarrow} 1
    \quad \quad \text{with $P_{\theta^*}$-probability $1$.}
\end{equation}
\end{theorem}
The requirement that $c \cdot \pi(\switchpar) \geq \pi(E_\switchpar)$
is automatically satisfied if $\pi$ is of the
form
\begin{equation}\label{eq:prior}
\pi(\switchpar)=\pim(m)\pik(k_1)\prod_{i=2}^{m}\pit(t_i|t_i>t_{i-1})\pik(k_i),
\end{equation}
where $\pim$, $\pik$ and $\pit$ are priors on $\posints$ with full
support, and $\pim$ is geometric:
$\pim(m)=\theta^{m-1}(1-\theta)$ for some $0\le\theta<1$. In this case $c=\theta/(1-\theta)$.

We now extend the theorem to the case where the universal
distributions $\m{p}_1, \m{p}_2, \ldots$ are not necessarily Bayesian,
i.e.\ they are not necessarily of the form (\ref{eq:bayesint}).
It turns out that the ``meta-Bayesian'' universal distribution
$\Pswitch$ is still consistent, as long as the following condition
holds. The condition essentially expresses that, for each $k$,
$\m{p}_k$ must not be too different from a Bayesian predictive
distribution based on (\ref{eq:bayesint}). This
can be verified if all models $\model_k$ are exponential families, and the $\m{p}_k$
represent ML or smoothed ML estimators (see Theorems~2.1 and 2.2 of
\citep{LiY00}). We suspect that it holds as well for more general
parametric models and universal codes, but we do not know of any proof.
\paragraph{Condition} There exist
Bayesian prediction strategies $\mbayes{p}_1, \mbayes{p}_2, \ldots$ of form (\ref{eq:bayesint}), 
with continuous and strictly positive priors $w_1, w_2, \ldots$ such that
\begin{enumerate}
\item The conditions of Theorem~\ref{thm:consistencyagain} hold for
  $\mbayes{p}_1, \mbayes{p}_2, \ldots$ and the chosen
  switch distribution prior $\pi$.
\item For all $k \in \posints$, for each compact subset $\Theta'$
  of the interior of $\Theta_k$, there exists a $K$ such that for all
  $\theta \in \Theta'$, with $P_\theta$-probability 1, for all $n$
$$
- \log \m{p}_k(X^n) + \log \mbayes{p}_k(X^n) \leq K.
$$
\item For all $k, k' \in \posints$ with $k \neq k'$ and all $x^n \in
\samplespace^*$, the distributions $\mbayes{P}_{k}(X_{n+1}^\infty \mid
x^n)$ and $\m{P}_{k'}(X_{n+1}^\infty \mid x^n)$ are mutually singular.\\
\end{enumerate}
\begin{theorem}[Consistency of the switch distribution, Part 2]
\label{thm:consistencyb}
Let $\m{p}_1, \m{p}_2, \ldots$ be prediction strategies and let $\pi$ be the 
prior of the corresponding switch distribution. 
Suppose that the condition above holds relative to $\m{p}_1, \m{p}_2, \ldots$
and $\pi$. Then, for all $k^* \in \posints$, for all $\theta^* \in \Theta_{k^*}$
except for a subset of $\Theta_{k^*}$ of Lebesgue-measure $0$, the
posterior distribution on $K_{n+1}$ satisfies
\begin{equation}
  \label{eqn:thmconsistencyb}
  \pi(K_{n+1} = k^* \mid X^n) \overset{n \to \infty}{\longrightarrow} 1
    \quad \quad \text{with $P_{\theta^*}$-probability $1$.}
\end{equation}
\end{theorem}

\section{Risk Convergence Rates}
\label{sec:nonparametrica}
In this section and the next we investigate how well the
switch distribution is able to predict future data in terms of
expected logarithmic loss or, equivalently, how fast estimates based
on the switch distribution converge to the true distribution in terms
of Kullback-Leibler risk. In Section~\ref{sec:prellie}, we define the
central notions of model classes, risk, convergence in Ces\`aro mean,
and minimax convergence rates, and we give the conditions on the prior
distribution $\pi$ under which our further results hold. We then
(Section~\ref{sec:switchvsbma}) show that the switch distribution
cannot converge any slower than standard Bayesian model averaging. As
a proof of concept, in Section~\ref{sec:histogram} we present
Theorem~\ref{thm:switchrsy}, which establishes that, in contrast to
Bayesian model averaging, the switch distribution converges at the
minimax optimal rate in a nonparametric histogram density estimation
setting.

In the more technical  Section~\ref{sec:nonparametricb}, we develop a
number of general tools for establishing optimal
convergence rates for the switch distribution, and we show that
optimal rates are achieved in, for example, nonparametric density
estimation with exponential families and (basic) nonparametric linear
regression, and also in standard parametric situations. 

\subsection{Preliminaries}
\label{sec:prellie}
\subsubsection{Model Classes}
\label{sec:modelclasses}
The setup is as follows. Suppose $\model_1,\model_2,\ldots$ is a
sequence of parametric models with associated estimators
$\Pbayes_1,\Pbayes_2,\ldots$ as before. Let us write $\model =
\union_{k=1}^\infty \model_k$ for the union of the models. Although
formally $\model$ is a set of prediction strategies, it will often be
useful to consider the corresponding set of distributions for $X^\infty
= (X_1, X_2, \ldots)$. With minor abuse of notation we will denote this
set by $\model$ as well.

To test the predictions of the switch distribution, we will want to
assume that $X^\infty$ is distributed according to a distribution $P^*$
that satisfies certain restrictions. These restrictions will always be
formulated by assuming that $P^* \in \model^*$, where $\model^*$ is some
restricted set of distributions for $X^\infty$. For simplicity, we will
also assume throughout that, for any $n$, the conditional distribution
$P^*(X_n \mid X^{n-1})$ has a density (relative to the Lebesgue or
counting measure) with probability one under $P^*$. For example, if
$\samplespace = [0,1]$, then $\model^*$ might be the set of all i.i.d.\
distributions that have uniformly bounded densities with uniformly
bounded first derivatives, as will be considered in
Section~\ref{sec:histogram}. In general, however, the sequence
$X^\infty$ need not be i.i.d.\ (under the elements of $\model^*$).

We will refer to any set of distributions for $X^\infty$ as a
\emph{model class}. Thus both $\model$ and $\model^*$ are model classes.
In Section~\ref{sec:parametric} it will be assumed that $\model^*
\subseteq \model$, which we will call the \emph{parametric} setting.
Most of our results, however, deal with various
\emph{nonparametric} situations, in which $\model^* \setminus \model$ is
non-empty. It will then be useful to emphasize that $\model^*$ is (much)
larger than $\model$ by calling $\model^*$ a \emph{nonparametric model
class}.

\subsubsection{Risk}

Given $X^{n-1} = x^{n-1}$, we will measure how well any estimator
$\m{P}$ predicts $X_n$ in terms of the Kullback-Leibler (KL) divergence
$D(P^*(X_n=\cdot \mid x^{n-1})\|\m{P}(X_n=\cdot \mid x^{n-1}))$
\citep{barron1998a}. Suppose that $P$ and $Q$ are distributions for some
random variable $Y$, with densities $p$ and $q$ respectively. Then the
KL divergence from $P$ to $Q$ is defined as
\begin{equation*}
  D(P\|Q) = E_P\left[\log \frac{p(Y)}{q(Y)}\right].
\end{equation*}
KL divergence is never negative, and reaches zero if and only if $P$
equals $Q$. Taking an expectation over $X^{n-1}$ leads to the standard
definition of the \emph{risk} of estimator $\m{P}$ at sample size $n$
relative to KL divergence:
\begin{equation}\label{eq:risk}
  \risk_n(P^*, \m{P}) = \mathop E_{X^{n-1} \sim P^*} \left[D\big(P^*(X_n = \cdot \mid
  X^{n-1}) \| \m{P}(X_n = \cdot \mid X^{n-1})\big) \right]. 
\end{equation}
Instead of the standard KL risk, we will study the {\em cumulative risk}
\begin{equation}
  \crisk_n(P^*, \m{P}) := \sum_{i=1}^n \risk_i(P^*,\m{P}),
\end{equation}
%
%
because of its connection to information theoretic redundancy (see e.g.\
\citep{barron1998a} or \citep[Chapter 15]{grunwald2007}): For all $n$ it
holds that
\begin{equation}\label{eq:redundancyrisk}
  \sum_{i=1}^n \risk_i(P^*,\m{P})
  = \sum_{i=1}^n E\left[\log \frac{p^*(X_i \mid X^{i-1})}{\m{p}(X_i
      \mid X^{i-1})}\right]
  = E\left[\log \prod_{i=1}^n \frac{p^*(X_i \mid X^{i-1})}{\m{p}(X_i
      \mid X^{i-1})}\right]
  = D\left(P^{*(n)} \| \Pbayes^{(n)}\right),
\end{equation}
where the superscript $(n)$ denotes taking the marginal of the
distribution on the first $n$ outcomes. We will show convergence of
the predictions of the switch distribution in terms
of the cumulative rather than the individual risk. This notion of
convergence, defined below, is equivalent to the well-studied notion of
convergence in Ces\`aro mean. It has been considered by, among others,
\citet{rissanen1992}, \citet{barron1998a}, \citet{PolandH05}, and its
connections to ordinary convergence of the risk were investigated in
detail by \citet{grunwald2007}. 

Asymptotic properties like `convergence' and `convergence in Ces\`aro
mean' will be expressed conveniently using the following notation, which
extends notation from \citep{YangB99}:
\begin{definition}[Asymptotic Ordering of Functions]
  For any two nonnegative functions $g,h: \posints \to \reals$ and any
  $c \geq 0$ we write
    $g \preceq_c h$
  if for all $\epsilon > 0$ there exists an $n_0$ such that for all $n
  \geq n_0$ it holds that $g(n) \leq (1+\epsilon) \cdot c \cdot h(n)$. The less
  precise statement that there exists {\em some\/}  $c >
  0$ such that $g \preceq_c \cdot h$, will be denoted by
    $g \preceq h$.
  (Note the absence of the subscript.) For $c > 0$, we define $h
  \succeq_c g$ to mean $g \preceq_{1/c} h$, and $h \succeq g$ means that
  for {\em some\/} $c > 0$, $h \succeq_c g$. Finally, we say that
    $g \asymp h$
  if both $g \preceq h$ and $h \preceq g$.
\end{definition}
Note that $g \preceq h$ is equivalent to $g(n) = O(h(n))$. 
One may think of $g(n) \preceq_c h(n)$ as another way of writing
$\limsup_{n \to \infty} g(n)/h(n) \leq c$. The two statements are
equivalent if $h(n)$ is never zero. 
%

We can now succinctly state that the risk of an estimator $\m{P}$
\emph{converges} to $0$ at rate $f(n)$ if $\risk_n(P^*,\m{P}) \preceq_1
f(n)$, where $f : \posints \to \reals$ is a nonnegative function such
that $f(n)$ goes to $0$ as $n$ increases. We say  that $\m{P}$ 
converges to 0 at rate at least $f(n)$ 
\emph{in Ces\`aro mean} if $\frac{1}{n}\sum_{i=1}^n \risk_i(P^*,\m{P})
\preceq_1 \frac{1}{n} \sum_{i=1}^n f(i)$. As $\preceq_1$-ordering is
invariant under multiplication by a positive constant, convergence in
Ces\`aro mean is equivalent to asymptotically bounding the cumulative
risk of $\m{P}$ as
\begin{equation}
\label{eq:cumulative}
\sum_{i=1}^n \risk_i(P^*,\m{P}) \preceq_1 \sum_{i=1}^n
f(i).
\end{equation}
We will always express convergence in Ces\`aro mean in
terms of cumulative risks as in (\ref{eq:cumulative}). The reader may verify that if the risk of
$\m{P}$ is always finite and converges to $0$ at rate $f(n)$ and
$\lim_{n \rightarrow \infty} \sum_{i=1}^n f(n) = \infty$, then the
risk of $\m{P}$ also converges in Ces\`aro mean at rate $f(n)$.
Conversely, suppose that the risk of $\m{P}$ converges in Ces\`aro
mean at rate $f(n)$. Does this also imply that the risk of $\m{P}$
converges to $0$ at rate $f(n)$ in the ordinary sense? The answer is
``almost'', as shown in \citep{grunwald2007}: The risk of $\m{P}$ may
be strictly larger than $f(n)$ for some $n$, but the \emph{gap}
between any two $n$ and $n'> n $ at which the risk of $\m{P}$ exceeds
$f$ must become infinitely large with increasing $n$. This indicates
that, although convergence in Ces\`aro mean is a weaker notion than
standard convergence, obtaining fast Ces\`aro mean convergence rates
is still a worthy goal in prediction and estimation. We explore the
connection between Ces\`aro and ordinary convergence in more detail in
Section~\ref{sec:standard}.
\subsubsection{Minimax Convergence Rates}
The worst-case cumulative risk of
the switch distribution is given by
\begin{equation}
  \Gsw(n) = \sup_{P^*\in\model^*}\sum_{i=1}^n \risk_i(P^*,\Pswitch).
  \label{eq:gsw}
\end{equation}
We will compare it to the
\emph{minimax cumulative risk}, defined as:
\begin{equation}
  \Gmmfix(n) := \inf_{\m{P}} \sup_{P^* \in \model^*}
    \sum_{i=1}^n  \risk_i(P^*, \m{P}),
  \label{eq:mmfixed}
\end{equation}
where the infimum is over all estimators $\m{P}$ as defined in
Section~\ref{sec:preliminaries}. We will say that the switch distribution achieves the \emph{minimax
convergence rate in Ces\`aro mean} (up to a multiplicative constant) if
$\Gsw(n) \preceq \Gmmfix(n)$.
Note that there
is no requirement that $\m{P}(X_{i+1} \mid x^i)$ is a distribution in
$\model^*$ or $\model$; We are looking at the worst case over all
possible estimators, irrespective of the model class, $\model$, used to
approximate $\model^*$. Thus, we may call $\m{P}$ an ``out-model
estimator'' \citep{grunwald2007}.

\subsubsection{Restrictions on the Prior}
\label{sec:restrictions}
Throughout our analysis of the achieved rate of convergence we will
require that the prior of the switch distribution, $\pi$, can be
factored as in \eqref{eq:prior}, and is chosen to satisfy
\begin{equation}
  -\log \pim(m) = O(m), \ \ 
  -\log \pik(k) = O(\log k), \ \ 
  -\log \pit(t) = O(\log t).
  \label{eq:mmprior}
\end{equation}
Thus $\pim$, the prior on the total number of distinct predictors, is
allowed to decrease either exponentially (as required for
Theorem~\ref{thm:consistencyagain}) or polynomially, but $\pit$ and
$\pik$ cannot decrease faster than polynomially. For example, we could
set $\pit(t) = 1/(t(t+1))$ and $\pik(k) = 1/(k(k+1))$, or we could
take the universal prior on the integers \citep{Rissanen83}.

\subsection{Never Much Worse than Bayes}
\label{sec:switchvsbma}
Suppose that the estimators $\Pbayes_1, \Pbayes_2, \ldots$ are
Bayesian predictive distributions, defined by their densities 
as in (\ref{eq:bayesint}). 
The following lemma expresses that the Ces\`aro mean of the risk
achieved by the switch distribution is never much higher than that of
Bayesian model averaging, which is itself never much higher than that of
any of the Bayesian estimators $\Pbayes_k$ under consideration.
\begin{lemma}\label{lem:switchvsbma}
  Let $\Pswitch$ be the switch distribution for
  $\Pbayes_1,\Pbayes_2,\ldots$ with prior $\switchprior$ of the
  form~\eqref{eq:prior}. Let $\Pbma$ be the Bayesian model averaging
  distribution for the same estimators, defined with respect to the
  same prior on the estimators $\pik$. Then, for all $n \in \posints$,
  all $x^n\in\samplespace^n$, and all $k\in\posints$,
  \begin{equation}
    \pswitch(x^n)
      \quad\geq\quad \pim(1)\pbma(x^n)
      \quad\geq\quad \pim(1)\pik(k)\pbayes_k(x^n).
  \end{equation}
  Consequently, if $X_1$, $X_2$, $\ldots$ are distributed according to
  any distribution $P^*$, then for any $k \in \posints$,
  \begin{equation}
    \sum_{i=1}^n \risk_i(P^*,\Pswitch)
    ~\leq~
    \sum_{i=1}^n \risk_i(P^*,\Pbma)-\log\pim(1)
    ~\leq~
    \sum_{i=1}^n \risk_i(P^*,\Pbayes_k)-\log\pim(1)-\log\pik(k).
  \end{equation}
\end{lemma}
As mentioned in the introduction, one advantage of model averaging
using $\pbma$ is that it always predicts almost as well as the
estimator $\m{p}_k$ for \emph{any} $k$, including the $\m{p}_k$ that
yields the best predictions overall. Lemma~\ref{lem:switchvsbma}
shows that this property is shared by $\pswitch$, which
multiplicatively dominates $\pbma$. In the sequel, we 
investigate under which circumstances the switch distribution may
achieve a \emph{smaller} cumulative risk than Bayesian model averaging.

\subsection{Histogram Density Estimation}
\label{sec:histogram}

How many bins should be selected in density estimation based on
histogram models with equal-width bins? Suppose $X_1$, $X_2$, $\ldots$
take outcomes in $\samplespace = [0,1]$ and are distributed i.i.d.\
according to $P^* \in \model^*$, where $P^*(X_n)$ has density $p^*$ for
all $n$. Let $p^*(x^n) = \prod_{i=1}^n p^*(x_i)$ for $x^n \in
\samplespace^n$. Let us restrict $\model^*$ to the set of distributions
with densities that are uniformly bounded above and below and also have
uniformly bounded first derivatives. In particular, suppose there exist
constants $0 < c_0 < 1 < c_1$ and $c_2$ such that
\begin{equation}
  \model^* = \left\{P^* \mid c_0 \leq p^*(x) \leq c_1 \text{ and }
                \left|d/dx ~ p^*(x)\right| \leq c_2
                \text{ for all $x \in [0,1]$} \right\}.
  \label{eqn:rsydensities}
\end{equation}

In this setting the minimax convergence rate in Ces\`aro mean can be
achieved using histogram models with bins of equal width (see below).
The equal-width histogram model with $k$ bins, $\model_k$, is specified
by the set of densities $\{p_\vec{\theta}\}$ on $\samplespace = [0,1]$
that are constant within the $k$ bins $[0,a_1]$, $(a_1,a_2]$, $\ldots$,
$(a_{k-1},1]$, where $a_i = i/k$. In other words, $\model_k$ contains
any density $p_\vec{\theta}$ such that, for all $x,x' \in [0,1]$ that
lie in the same bin, $p_\vec{\theta}(x) = p_\vec{\theta}(x')$. The
$k$-dimensional parameter vector $\vec{\theta} = (\theta_1, \ldots,
\theta_k)$ denotes the probability masses of the bins, which have to sum
up to one: $\sum_{i=1}^k \theta_i = 1$. Note that this last constraint
makes the number of degrees of freedom one less than the number of bins.
Following \citet{yu1992} and \citet{rissanen1992}, we associate the
following estimator $\m{p}_k$ with model $\model_k$:
\begin{equation}
\label{eq:laplace}
  \m{p}_k(X_{n+1}\mid x^n) := \frac{n_{X_{n+1}}(x^n) + 1}{n + k}\cdot k,
\end{equation}
where $n_{X_{n+1}}(x^n)$ denotes the number of outcomes in $x^n$ that
fall into the same bin as $X_{n+1}$. As in Example~\ref{ex:bern}, these
estimators may both be interpreted as being based on parameter
estimation (estimating $\theta_i = (n_i(x^n) + 1)/(n+k)$, where
$n_i(x^n)$ denotes the number of outcomes in bin $i$) or on Bayesian
prediction (a uniform prior for $\vec{\theta}$ also leads to this
estimator \citep{yu1992}).

The minimax convergence rate in Ces\`aro mean for $\model^*$ is of the
order of $n^{-2/3}$ \citep[Theorems~3.1 and 4.1]{yu1992}\footnote{We note
that \citep{yu1992} reproduces part of Theorem~1 from
\citep{rissanen1992} without the (necessary) condition that $c_0 < 1 < c_1$.},
which is equivalent to the statement that
\begin{equation}
  \Gmmfix(n) \preceq n^{1/3}.
\end{equation}
This rate is achieved up to a multiplicative constant by the model
selection criterion $\delta(x^n) = \ceil{n^{1/3}}$, which,
irrespective of the observed data, uses the
histogram model with $\ceil{n^{1/3}}$ bins to predict $X_{n+1}$
\citep{rissanen1992}:
\begin{equation}
 \sup_{P^* \in \model^*} \sum_{i=1}^n
  \risk_i(P^*,\m{P}_{\ceil{i^{1/3}}}) \preceq n^{1/3}.
  \label{eqn:rsy1}
\end{equation}
The optimal rate in Ces\`aro mean is also achieved (up to a
multiplicative constant) by the switch distribution:
\begin{theorem}
  \label{thm:switchrsy}
  Let $\m{p}_1$, $\m{p}_2$, $\ldots$ be histogram estimators as in
  \eqref{eq:laplace}, and let $\pswitch$ denote the switch distribution
  relative to these estimators with prior that satisfies the conditions
  in \eqref{eq:mmprior}. Then
  \begin{equation}
 \Gsw(n) =     \sup_{P^* \in \model^*} \sum_{i=1}^n \risk_i(P^*,\Pswitch) \preceq n^{1/3}.
  \end{equation}
\end{theorem}

\subsubsection{Comparison of the Switch Distribution to Other Estimators}

To return to the question of choosing the number of histogram bins, we
will now first compare the switch distribution to the minimax optimal
model selection criterion $\delta$, which selects $\ceil{n^{1/3}}$ bins.
We will then also compare it to Bayes factors model selection and
Bayesian model averaging.

Although $\delta$ achieves the minimax convergence rate in Ces\`aro
mean, it has two disadvantages compared to the switch distribution: The
first is that, in contrast to the switch distribution, $\delta$ is
inconsistent. For example, if $X_1, X_2, \ldots$ are i.i.d.\ according
to the uniform distribution, then $\delta$ still selects $\lceil
n^{1/3} \rceil$ bins, while  model selection based on 
$\Pswitch$ will correctly select the 1-bin histogram for all large
$n$. Experiments with simulated data confirm that $\Pswitch$ already prefers the 1-bin histogram
at quite small sample sizes. 
The other disadvantage is that if we are lucky enough to be in a
scenario where $P^*$ actually allows a \emph{faster} than the minimax
convergence rate by letting the number of bins grow as $n^\gamma$ for
some $\gamma \neq \frac{1}{3}$, the switch distribution would be able to
take advantage of this whereas $\delta$ cannot. Our experiments with
simulated data confirm that, if $P^*$ has a sufficiently smooth
density, then it predictively outperforms $\delta$ by a wide margin.  

To achieve consistency one might also construct a Bayesian estimator
based on a prior distribution on the number of bins. However,
\cite[Theorem~2.4]{yu1992} suggests that Bayesian model averaging
does \emph{not} achieve the same rate\footnote{In the left-hand side of
(iii) in Theorem~2.4 of \citep{yu1992} the division by $n$ is missing.
(See its proof on p.~203 of that paper.)}, but a rate of order
$n^{-2/3}(\log n)^{2/3}$ instead, which is equivalent to the statement
that
\begin{equation}
\label{eq:bayessup}
  \sup_{P^* \in \model^*} \sum_{i=1}^n \risk_i(P^*,\Pbma)
    \asymp n^{1/3}(\log n)^{2/3}.
\end{equation}
Bayesian model averaging will typically predict better than the single
model selected by Bayes factor model selection. We should therefore
not expect Bayes factor model selection to achieve the minimax rate
either. While we have no formal proof that standard Bayesian model
averaging behaves like (\ref{eq:bayessup}), we have also performed
numerous empirical experiments which all confirm that Bayes performs
significantly worse than the switch distribution. We will report on
these and the other aforementioned experiments elsewhere.

What causes this Bayesian inefficiency? Our explanation is that, as the
sample size increases, the catch-up phenomenon occurs at each time
that switching to a larger number of bins is required. Just like in the
shaded region in Figure~\ref{fig:markovexample}, this causes Bayes to
make suboptimal predictions for a while after each switch. This
explanation is supported by the fact that the switch distribution, which
has been designed with the catch-up phenomenon in mind, does not suffer
from the same inefficiency, but achieves the minimax rate in Ces\`aro
mean.

%
\section{Risk Convergence Rates, Advanced Results}
\label{sec:nonparametricb}
In this section we develop the theoretical results needed to prove
minimax convergence results for the switch distribution. First, in
Section~\ref{sec:oracle}, we define the convenient concept of an
oracle and show that the switch distribution converges at least as
fast as oracles that do not switch too often as the sample size
increases. In order to extend the oracle results to convergence rate
results, it is useful to restrict ourselves to model classes
$\model^*$ of the ``standard'' type that is usually considered in the
nonparametric literature. Essentially, this amounts to imposing an
independence assumption and the assumption that the convergence rate is
of order at least $n^{- \gamma}$ for some $\gamma <1$.  In
Section~\ref{sec:standard} we define such standard nonparametric
classes formally, we explain in detail how their Ces\`aro convergence
rate relates to their standard convergence rate, and we provide our
main lemma, which shows that, for standard nonparametric classes,
$\Pswitch$ achieves the minimax rate under a rather weak condition.
In Section~\ref{sec:exponential} and~\ref{sec:linear} we apply this
lemma to show that $\Pswitch$ achieves the minimax rates in some
concrete nonparametric settings: density estimation based on
exponential families and linear regression.  Finally,
Section~\ref{sec:parametric} briefly considers the parametric case.

To get an intuitive idea of how the switch distribution avoids the
catch-up phenomenon, it is essential to look
at the proofs of some of the results in this section, in particular Lemma~\ref{lem:simulateoracle},
\ref{lem:rateofconvergence}, \ref{standardnonparametric},
\ref{lem:linreg} and~\ref{lem:parametric}. Therefore,
the proofs of these lemmas have been kept in the main text. 
\subsection{Oracle Convergence Rates}
\label{sec:oracle}
Let $\model^*$, $\model_1$, $\model_2$, $\ldots$ and $\m{P}_1$,
$\m{P}_2$, $\ldots$  be as in Section~\ref{sec:modelclasses}. As a technical
tool, it will be useful to compare the cumulative risk of the
switch distribution to that of an \emph{oracle} prediction strategy that
knows the true distribution $P^* \in \model^*$, but is restricted to
switching between $\m{P}_1$, $\m{P}_2$, $\ldots$.
Lemma~\ref{lem:simulateoracle} below gives an upper bound on the
additional cumulative risk of the switch distribution compared to such
an oracle. To bound the rate of convergence in Ces\`aro mean for various
nonparametric model classes we also formulate
Lemma~\ref{lem:rateofconvergence}, which is a direct consequence of
Lemma~\ref{lem:simulateoracle}. Lemma~\ref{lem:rateofconvergence} will serve as a basis for further rate of convergence results in
Sections~\ref{sec:standard}--\ref{sec:linear}.
\begin{definition}[Oracle]
\label{def:oracle}
An oracle is a function $\oracle:\model^*\times
\Union_{n=0}^\infty \samplespace^n \to\posints$ that, for all $n \in
\naturals$, given not only the observed data $x^n \in \samplespace^n$,
but also the true distribution $P^*\in\model^*$, selects a model index,
$\oracle(P^*,x^n)$, with the purpose of predicting $X_{n+1}$ by
$P_\oracle(X_{n+1} \mid x^n) \equiv \m{P}_{\oracle(P^*,x^n)}(X_{n+1}
\mid x^n)$.
\end{definition}
If $\oracle(P^*,x^n)=\oracle(P^*,y^n)$ for any $x^n,y^n \in
\samplespace^n$ (i.e.\ the oracle's choices do not depend on $x^n$, but
only on $n$), we will say that oracle $\oracle$ \emph{does not look at
the data} and write $\oracle(P^*,n)$ instead of $\oracle(P^*,x^n)$ for
some arbitrary $x^n \in \samplespace^n$.

Suppose $\oracle$ is an oracle and $X_1$, $X_2$, $\ldots$ are distributed
according to $P^* \in \model^*$. If $X^{n-1} = x^{n-1}$, then
$\oracle(P^*,x^0),\ldots,\oracle(x^{n-1})$ is the sequence of model
indices chosen by $\oracle$ to predict $X_1$, $\ldots$, $X_n$. We may
split this sequence into segments where the same model is chosen. Let us
define $m_\oracle(n)$ as the maximum number of such distinct segments
over all $P^* \in \model^*$ and all $x^{n-1} \in \samplespace^{n-1}$.
That is, let
%
\begin{equation}\label{eq:mn}
  m_\oracle(n) = \max_{P^* \in \model^*}
  \max_{x^{n-1} \in \samplespace^{n-1}} |\{1 \leq i \leq n-1 :
  \oracle(P^*,x^{i-1}) \neq \oracle(P^*,x^i)\}| + 1,
\end{equation}
where $x^i$ denotes the prefix of $x^{n-1}$ of length $i$. (The maximum
always exists, because for any $P^*$ and $x^{n-1}$ the number of
segments is at most $n$.)

The following lemma expresses that any oracle $\oracle$ that does not
select overly complex models, can be approximated by the
switch distribution with a maximum additional risk that depends on
$m_\oracle(n)$, its maximum number of segments.  We will typically be
interested in oracles $\oracle$ such that this maximum is small in
comparison to the sample size, $n$. The lemma is a tool in establishing
the minimax convergence rates of $\Pswitch$ that we consider in the
following sections.

\begin{lemma}[Oracle Approximation Lemma]
  \label{lem:simulateoracle}
  Let $\Pswitch$ be the switch distribution, defined with respect to a
  sequence of estimators $\m{P}_1,\m{P}_2,\ldots$ as introduced above,
  with any prior $\switchprior$ that satisfies the conditions in
  \eqref{eq:mmprior} and let $P^*\in\model^*$. Suppose $g : \posints \to
  \reals$ is a positive, nondecreasing function and $\oracle$ is an
  oracle such that
  \begin{equation}
    \oracle(P^*,x^{i-1}) \leq g(i)
    \label{eqn:simplemodels}
  \end{equation}
  for all $i \in \posints$, all $x^{i-1} \in \samplespace^{i-1}$. Then
  \begin{equation}
    \sum_{i=1}^n \risk_i(P^*,\Pswitch) = \sum_{i=1}^n \risk_i(P^*,P_\oracle) +
    O\big(m_\oracle(n)\cdot(\log n + \log g(n))\big),
    \label{eqn:codingswitchpoints}
  \end{equation}
  where the constants in the big-O notation depend only on the constants
  implicit in \eqref{eq:mmprior}.
\end{lemma}

\begin{proof}
Using~\eqref{eq:redundancyrisk} we can rewrite
  \eqref{eqn:codingswitchpoints} into the equivalent claim
  \begin{equation}\label{eqn:equivalent}
    E\left[ \log \frac{p_\oracle(X^n)}{\pswitch(X^n)} \right]
      = O\big(m_\oracle(n)\cdot(\log n + \log g(n))\big),
  \end{equation}
  which we will proceed to prove. For all $n$, $x^n\in\samplespace^n$,
  there exists an $\switchpar\in\switchpars$ with $m(\switchpar)\leq
  m_\oracle(n)$ and $t_{m(\switchpar)}(\switchpar) < n$ that selects the
  same sequence of models as $\oracle$ to predict $x^n$, so that
  $q_\switchpar(x^i\mid x^{i-1})= p_{\oracle_i}(x^i\mid x^{i-1})$ for
  $1\leq i\leq n$. Consequently, we can bound
  \begin{equation}
    \pswitch(x^n)~=~\sum_{\switchpar' \in \switchpars} q_{\switchpar'}(x^n)
        \cdot \switchprior(\switchpar')
        ~\geq~ q_{\switchpar}(x^n)
        \switchprior(\switchpar)
        ~=~ p_\oracle(x^n)
        \switchprior(\switchpar).
    \label{eqn:simoracle}
  \end{equation}
  By assumption \eqref{eqn:simplemodels} we have that $\oracle$, and
  therefore $\switchpar$, never selects a model $\model_k$ with index
  $k$ larger than $g(i)$ to predict the $i$\kern.3pt th outcome.
  Together with $\eqref{eq:mmprior}$ and the fact that $g$ is
  nondecreasing, this implies that
  \begin{align}
    -\log \switchprior(\switchpar) 
      &= -\log \pim(m(\switchpar)) -\log \pik(k_1(\switchpar)) +
        \sum_{j=2}^{m(\switchpar)} -\log \pit\big(t_j(\switchpar)\mid
        t_{j-1}(\switchpar)\big) -\log \pik(k_j(\switchpar)) \notag \\
      &= O(m(\switchpar)) + \sum_{j=1}^{m(\switchpar)} O\big(\log
        t_j(\switchpar)\big) + O\big(\log k_j(\switchpar)\big) \notag \\
      &= O(m(\switchpar)) + \sum_{j=1}^{m(\switchpar)}
        O\big(\log t_j(\switchpar)\big)
        + O\big(\log g(t_j(\switchpar)+1)\big) \notag \\
      &= O(m(\switchpar)) + \sum_{j=1}^{m(\switchpar)}
        O(\log n) + O\big(\log g(n)\big)
      = O\big(m_\oracle(n)\cdot(\log n + \log g(n))\big),
    \label{eqn:swpriorbound}
  \end{align}
  where the constants in the big-O in the final expression depend only
  on the constants in~\eqref{eq:mmprior}. Together \eqref{eqn:simoracle}
  and \eqref{eqn:swpriorbound} imply \eqref{eqn:equivalent}, which was
  to be shown.
\end{proof}
%

From an information theoretic point of view, the additional risk of the
switch distribution compared to oracle $\oracle$ may be interpreted as
the number of bits required to encode how the oracle switches between
models.

In typical applications, we use oracles that achieve the minimax rate,
and that are such that the number of segments $m_{\oracle}(n)$ is
logarithmic in $n$, and $\oracle$ never selects a model index larger
than $n^{\tau}$ for some $\tau > 0$ (typically, $\tau \leq 1$ but some
of our results allow larger $\tau$ as well). By
Lemma~\ref{lem:simulateoracle}, the additional risk of the switch
distribution over such an oracle is $O((\log n)^2)$. In nonparametric
settings, the minimax rate satisfies $\Gmmfix(n) \succeq n^{1- \gamma}$
for some $\gamma < 1$. This indicates that, for large $n$, the
additional risk of the switch distribution over a sporadically
switching oracle becomes negligible. This is the basic idea that
underlies the nonparametric minimax convergence rate results of
Section~\ref{sec:standard}-\ref{sec:linear}. Rather than using
Lemma~\ref{lem:simulateoracle} directly to prove such results, it is
more convenient to use its straightforward extension
Lemma~\ref{lem:rateofconvergence} below, which  bounds the
worst-case cumulative risk of the switch distribution in terms of the
worst-case cumulative risk of an oracle, $\oracle$:
\begin{equation}
  G_\oracle(n) = \sup_{P^*\in\model^*}\sum_{i=1}^n \risk_i(P^*,P_\oracle).
  \label{eq:gsig}
\end{equation}

\begin{lemma}[Rate of Convergence Lemma]
\label{lem:rateofconvergence}
  Let $\Pswitch$ be the switch distribution, defined with respect to a
  sequence of estimators $\m{P}_1,\m{P}_2,\ldots$ as above, with any
  prior $\pi$ that satisfies the conditions in \eqref{eq:mmprior}. Let
  $f : \posints \to \reals$ be a nonnegative function and let
  $\modelclass$ be a set of distributions on $X^\infty$. Suppose there
  exist a positive, nondecreasing function $g : \posints \to \reals$, an
  oracle $\oracle$, and constants $c_1,c_2 \geq 0$ such that
  \begin{enumerate}[(i)]
  \item 
$\displaystyle\oracle(P^*,x^{i-1}) \leq g(i)$
  \quad (for all $i \in \posints$, $x^{i-1} \in \samplespace^{i-1}$, and $P^*
  \in \model^*$),
  \label{cond:canapplypreviouslemma}
  \item 
$ m_\oracle(n)\big(\log n + \log  g(n)\big) \preceq_{c_2} f(n)$\label{cond:ii}
  \item 
$ G_\oracle(n) \preceq_{c_1} f(n)$ \label{cond:iii}
  \end{enumerate}
  Then there exists a constant $c_3 > 0$ such that
$\Gsw(n) \preceq_{c_1 + c_2 \cdot c_3} f(n)$.
\end{lemma}
\begin{proof}
By Lemma~\ref{lem:simulateoracle} we have that $\Gsw(n) =
  G_\oracle(n) + O\big(m_\oracle(n)\cdot(\log n + \log g(n))\big)$.
  Therefore there exists a constant $c_3 > 0$ such that
  \begin{equation*}
    \limsup_{n \to \infty} \frac{\Gsw(n)}{f(n)}
      \leq
    \limsup_{n \to \infty} \frac{G_\oracle(n) + 
      c_3 \cdot m_\oracle(n)\cdot(\log n + \log g(n))}{f(n)}
    \leq c_1 + c_2 \cdot c_3,
  \end{equation*}
  where the second inequality follows from Conditions~\ref{cond:ii} en
  \ref{cond:iii}. 
\end{proof}
Note that Condition~\ref{cond:ii} is satisfied with $c_2 = 0$ iff
$m_\oracle(n)\big(\log n + \log g(n)\big) = o(f(n))$.
In the following subsections, we prove that $\Pswitch$ achieves
$\Gmmfix(n)$ relative to various parametric and nonparametric model
classes $\model^*$ and $\model$. The proofs are invariably based on
applying Lemma~\ref{lem:rateofconvergence}. Also, the proof of
Theorem~\ref{thm:switchrsy} is based on Lemma~\ref{lem:rateofconvergence}. The general idea is to
apply the lemma with $f(n)$ equal to the summed minimax risk 
$\Gmmfix(n)$ (see~\eqref{eq:mmfixed}). If, for a given model class
$\model^*$, one can exhibit an oracle $\oracle$ that only switches
sporadically (Condition (ii) of the lemma) and that achieves
$\Gmmfix(n)$ (Condition (iii)), then the lemma implies that $\Pswitch$
achieves the minimax rate as well.
\subsection{Standard Nonparametric Model Classes}
\label{sec:standard}
 In this section we define ``standard nonparametric model classes'',
and we present our main lemma, which shows that, for such classes,
$\Pswitch$ achieves the minimax rate under a rather weak condition.
Standard nonparametric classes are defined in terms of the (standard,
non-Ces\`aro) minimax rate. Before we give a precise definition of
standard nonparametric, it is useful to compare the standard rate to
the Ces\`aro-rate. For given $\model^*$, the standard minimax rate is
defined as
\begin{equation}
  \label{eq:mmstandard}
  \gmm(n) = \inf_{\m{P}} \sup_{P^* \in \model^*} \risk_n(P^*, \m{P}),
\end{equation}
where the infimum is over all possible estimators, as defined in
Section~\ref{sec:preliminaries}; $\m{P}$  is not required to lie in
$\model^*$ or $\model$. If an estimator achieves (\ref{eq:mmstandard})
to within a constant factor, we say that it converges at the
\emph{minimax optimal rate}. Such an estimator will also achieve the
minimax cumulative risk for varying $P^*$, defined as
\begin{equation}
  \label{eq:mmvar}
  \Gmmvar(n) = \sum_{i=1}^n \gmm(i) = \inf_{\m{P}} 
  \ \sum_{i=1}^n \sup_{P^* \in \model^*} \risk_i(P^*, \m{P}),
\end{equation}
where the infimum is again over all possible estimators.

In many nonparametric density estimation and
regression problems, the minimax risk $\gmm(n)$ is of order $n^{-
\gamma}$ for some $1/2 < \gamma <1$ (see, for example,
\citep{YangB98,YangB99,BarronS91}), i.e.\ $\gmm(n) \asymp n^{- \gamma}$,
where $\gamma$ depends on the smoothness assumptions on the densities in
$\model^*$. In this case, we have
\begin{equation}
  \label{eq:sumint}
  \Gmmvar(n) \asymp \sum_{i=1}^n i^{-\gamma}
             \asymp \int_1^n x^{-\gamma}
             \dif x \asymp n^{1 - \gamma}.
\end{equation}
Similarly, in standard parametric problems, the minimax risk $\gmm(n)
\asymp 1/n$. In that case, analogously to (\ref{eq:sumint}), we see that
the minimax cumulative risk $\Gmmvar$ is of order $\log n$.

Note, however, that our previous result for histograms (and, more generally, all results we are about to present), is based on a scenario where $P^*$, while
allowed to depend on $n$, is kept fixed over the terms in the sum from
$1$ to $n$.  Indeed, in Theorem~\ref{thm:switchrsy} we showed that
$\Pswitch$ achieves the minimax rate  $\Gmmfix(n)$ as defined in
(\ref{eq:mmfixed}).  Comparing to (\ref{eq:mmvar}), we see that the
supremum is moved outside of the sum.  Fortunately, $\Gmmfix$ and
$\Gmmvar$ are usually of the same order: in the parametric case, e.g.
$\model^* = \bigcup_{k \leq k^*} \model_k$, both $\Gmmfix$ and $\Gmmvar$
are of order $\log n$.  For $\Gmmvar$, we have already seen this.  For
$\Gmmfix$, this is a standard information-theoretic result, see for
example \citep{ClarkeB90}. In a variety of  standard  nonparametric
situations that are studied in the literature, we have $\Gmmvar(n)
\asymp \Gmmfix(n)$ as well. Before showing this, we first define what we
mean by ``standard nonparametric situations'':
\begin{definition}[Standard Nonparametric]
\label{def:standardnonpar}
We call a model class $\model^*$ \emph{standard
nonparametric} if
\begin{enumerate}
  \item For any $P^* \in \model^*$, the random variables $X_1$, $X_2$,
  $\ldots$ are independent and identically distributed whenever
  $X^\infty \sim P^*$, and $P^*(X_1)$ has a density (relative to the
  Lebesgue or counting measure); and
  \item The minimax convergence rate, $\gmm(n)$, relative to $\model^*$
    does not decrease too fast in the sense that, for some $0 < \gamma < 1$, some nondecreasing
function $h_0(n) = o(n^{\gamma})$, it holds that 
    \begin{equation}
      \label{eq:hgrowth}
      \gmm(n) \asymp n^{-\gamma} h_0(n).
    \end{equation} 
\end{enumerate}
\end{definition}
Examples of standard nonparametric $\model^*$ include cases with
$\gmm(n) \asymp n^{- \gamma}$ (in that case $h_0(n) \equiv 1$), or,  more generally, $\gmm(n) \asymp n^{-
\alpha} (\log n)^{\beta}$ for some $\alpha \in (0,1), \beta \in
\reals$ (take $\gamma > \alpha$ and $h_0(n) = n^{\gamma - \alpha}
(\log n)^\beta$; note that $\beta$ may be negative); see
\citep{YangB99}. While in Lemma~\ref{lem:rateofconvergence} there are neither
independence nor convergence rate assumptions, in the next section we
develop extensions of Lemma~\ref{lem:rateofconvergence} and
Theorem~\ref{thm:switchrsy} that do restrict attention to such
``standard nonparametric'' model classes.

\begin{proposition}
\label{prop:varfix}
  For all standard nonparametric model classes, it holds that
  $\Gmmfix(n) \asymp \Gmmvar(n)$.
\end{proposition}
Summarizing, both in standard parametric and nonparametric cases,
$\Gmmfix$ and $\Gmmvar$ are of comparable size. Therefore,
Lemma~\ref{lem:simulateoracle} and \ref{lem:rateofconvergence} do suggest
that, both in standard parametric and nonparametric cases, $\Pswitch$
achieves the minimax convergence rate $\Gmmfix$. In particular, this
will hold if there exists an oracle $\oracle$ which achieves the
minimax convergence rate, but which, at the same time, switches only
sporadically.  However, the existence of such an oracle is often hard
to show directly. Rather than applying
Lemma~\ref{lem:rateofconvergence} directly, it is therefore often more
convenient to use Lemma~\ref{standardnonparametric} below, whose proof
is based on Lemma~\ref{lem:rateofconvergence}.
Lemma~\ref{standardnonparametric} gives a sufficient condition for
achieving the minimax rate that is easy to establish for several
standard nonparametric model classes: If there exists an oracle
$\oracle$ that achieves the minimax rate, such that all oracles
$\oracle'$ that lag a little behind $\oracle$ achieve the minimax rate
as well, then $\Pswitch$ must achieve the minimax rate as well.  Here
``lags a little behind'' means that the model chosen by $\oracle'$ at
sample size $n$ was chosen by $\oracle$ at a somewhat earlier sample
size. Formally, we fix some constants $\alpha > 1$ and $c > 0$.
Suppose that, for some oracles $\oracle$ and $\oracle'$, we have, for
all $P^* \in \model^*$, $n \in \posints$ and $x^{n-1} \in
\samplespace^{n-1}$,
$$\oracle'(P^*,x^{n-1}) \in
  \left\{\oracle(P^*,x^{i-1}) \mid i \in [n/\alpha,n] \intersection
  \naturals \right\}, $$ 
where $x^{i-1}$ denotes the prefix
  of $x^{n-1}$ of length $i-1$. In such a case we say that $\oracle'$
  {\em lags behind $\oracle$ by at most a factor of $\alpha$}. Intuitively this
  means that, at each sample size $n$, $\oracle'$ may choose any of the
  models that was chosen by $\oracle$ at sample size between $n/\alpha$
  and $n$. We call an oracle $\oracle$ {\em finite\/} relative to $\model^*$ if for all $n$, $\sup_{P^* \in \model^*} \risk_n(P^*, P_{\oracle}) < \infty$.

\begin{lemma}[Standard Nonparametric Lemma]
  \label{standardnonparametric}
  Suppose $\m{P}_1$, $\m{P}_2$, $\ldots$ are estimators and $\Pswitch$
  is the corresponding switch distribution with prior that satisfies
  \eqref{eq:mmprior}.
  Let $\model^*$ be a standard nonparametric model class. Let $\tau
  > 0$ be a constant, and let $\oracle$ be an oracle such that
  $\oracle(P^*,x^{n-1}) \leq n^\tau$ for all $P^* \in \model^*$, $n \in
  \posints$ and $x^{n-1} \in \samplespace^{n-1}$. Suppose that any
  oracle $\oracle'$ that lags behind $\oracle$ by at most a factor of
  $\alpha > 1$, is finite relative to $\model^*$, and achieves the minimax convergence rate up to a
  multiplicative constant $c > 0$:
  \begin{equation}
\sup_{P^* \in \model^*} \risk_n(P^*,{P}_{\oracle'}) \preceq_{c} \gmm(n).
    \label{eqn:instantminmax}
  \end{equation}
  Then the switch distribution achieves the minimax risk in Ces\`aro mean up
  to a multiplicative constant: 
  \begin{equation}
    \label{eqn:switchminmax}
\Gsw(n) \preceq_{c \cdot c'} \Gmmfix(n),
  \end{equation}
  where $c' = \limsup_{n \to \infty} \frac{\Gmmvar(n)}{\Gmmfix(n)}$.
\end{lemma}
\begin{proof}
  Let $t_j = \ceil{\alpha^{j-1}}-1$ for $j\in\posints$ be a sequence of
  switch-points that are exponentially far apart, and define an oracle
  $\oracle'$ as follows: For any $n \in \posints$, find $j$ such that
  $n\in[t_j+1,t_{j+1}]$ and let $\oracle'(P^*,x^{n-1}) :=
  \oracle(P^*,x^{t_j})$ for any $P^* \in \model^*$ and any $x^{n-1} \in
  \samplespace^{n-1}$. If we can apply Lemma~\ref{lem:rateofconvergence}
  for oracle $\oracle'$, with $f(n)=\Gmmfix(n)$, $g(n) = n^\tau$, $c_1 =
  c\cdot c'$ and $c_2 = 0$, we will obtain \eqref{eqn:switchminmax}. It
  remains to show that in this case conditions (i)--(iii) of
  Lemma~\ref{lem:rateofconvergence} are satisfied.
  
  As to condition \eqref{cond:canapplypreviouslemma}:
  $\oracle'(P^*,x^{n-1}) = \oracle(P^*,x^{t_j}) \leq (t_j+1)^\tau \leq
  n^\tau$. Condition \eqref{cond:ii} is also satisfied, because
  $m_{\oracle'}(n) \leq \ceil{\log_a n} + 2$, which implies
  \begin{eqnarray*}
    m_{\oracle'}(n)(\log n + \log g(n))
      &\leq& (\ceil{\log_a n}+2) (\log n + \log n^\tau) \\
      &\preceq& (\log n)^2 \preceq_0 \, n^{1 - \gamma}
      \asymp \Gmmvar(n) \asymp \Gmmfix(n)
  \end{eqnarray*}
  for some $\gamma < 1$, where we used that, because $\model^*$ is
  standard nonparametric, both \eqref{eq:sumint} and
  Proposition~\ref{prop:varfix} hold.   To verify condition
  \eqref{cond:iii}, first note that by choice of the switch-points,
  $\oracle'(P^*,x^{n-1}) = \oracle(P^*,x^{t_j})$ with $t_j+1 \in
  [n/\alpha,n]$ and therefore $\oracle'$ satisfies
  \eqref{eqn:instantminmax} by assumption. Since $\oracle'$ is finite relative to $\model^*$, this implies that ${G_{\oracle'}(n)}
  \preceq_c {\sum_{i=1}^n \gmm(i)} = \Gmmvar(n)$ and hence that
  \begin{equation*}
    G_{\oracle'}(n) \preceq_c {\Gmmfix(n)} \frac{\Gmmvar(n)}{\Gmmfix(n)}
    \preceq_{c \cdot c'} \Gmmfix(n).\qedhere
  \end{equation*}
\end{proof}
\subsection{Example: Nonparametric Density Estimation with Exponential Families}
\label{sec:exponential}
In many nonparametric situations, there exists an oracle 
$\oracle$ that achieves the minimax convergence rate, which 
only selects a model based on the sample size an not on the observed
data. This holds, for example, for
density estimation based on sequences of exponential families as
introduced by \citet{BarronS91,Sheu90} under the
assumption that the log density of the true distribution is in a
Sobolev space.  Not surprisingly, using Lemma~\ref{standardnonparametric},
we can show that $\Pswitch$ achieves the minimax rate in the
Barron-Sheu setting.

Formally, let $\samplespace = [0,1]$, let $r \geq 1$ and let $W^r_2$
be the Sobolev space of functions $f$ on ${\cal X}$ for which
$f^{(r-1)}$ is absolutely continuous and $\int (f^{(r)}(x))^2\dif x$
is finite. Here $f^{(r)}$ denotes the $r$-th derivative of $f$. Let
$\model^{*(r)}$ be the model class such that for any $P^* \in
\model^*$ the random variables $X_1, X_2, \ldots$ are i.i.d., and
$P^*(X_1)$ has a density $p^*$ such that $\log p^* \in W^r_2$. We
model $\model^{*(r)}$ using sequences of exponential families
$\model_1, \model_2, \ldots$ defined as follows. Let $\model_k = \{
p_{\theta} \mid \theta \in \reals^k \}$ be the $k$-dimensional
exponential family of densities on $[0,1]$ with
$$
p_{\theta}(x)  = p_0(x) \exp \left\{ \sum_{j=1}^k \theta_k \phi_k (x)
  - \psi_k(\theta) \right\},
$$ where $\psi_k(\theta) = \log \int p_0(x) \exp (\sum \theta_k
  \phi_k(x))\dif x$. Here $p_0$ is some reference density on $[0,1]$,
  taken with respect to Lebesgue measure. The density $p_{\theta}$ is
  extended to $X^{\infty}$ by independence. We let $\phi_1, \phi_2,
  \ldots$ be a countably infinite list of uniformly bounded, linearly
  independent functions, and we define $S_k := \{1, \phi_1, \ldots,
  \phi_k \}$. We consider three possible choices for $S_k$:
  polynomials, trigonometric series and splines of order $s$ with
  equally spaced knots. For example, we are allowed to choose $1, x,
  x^2, \ldots$ in the polynomial case, or $1$, $\cos(2 \pi x)$, $\sin
  (2 \pi x)$, $\ldots$, $\cos(2\pi (k/2) x)$, $\sin (2 \pi(k/2) x)$ in
  the trigonometric case. For precise conditions on the $\phi_1,
  \ldots, \phi_k$ that are allowed in each case, we refer to
  \citep{BarronS91}. We equip $\model_k$ with a Gaussian prior density
  $w_k(\theta)$, i.e.\ the parameters $\theta \in \reals_k$ are
  independent Gaussian random variables with mean 0 and a fixed
  variance $\sigma^2$. With each $\model_k$ we associate the Bayesian
  MAP estimator $p_{\hat{\theta}_k(x^n)}$, where $\hat{\theta}_k(x^n)
  := \arg \max_{\theta \in \reals^k} p_{k,\theta} (x^n)
  w_k(\theta)$. Define the corresponding prediction strategy $\m{P}_k$
  by its density $\m{p}_k(x_{n+1} \mid x^n) :=
  p_{\hat{\theta}_k(x^n)}(x_{n+1})$. Theorem 3.1 of \citep{Sheu90} (or
  rather its corollary on page~50 of \citep{Sheu90}) states the
  following:
\begin{theorem}[Barron and Sheu]
  \label{thm:barronsheu}
  Let $\phi_1, \phi_2, \ldots$ constitute a basis of polynomials, or
  trigonometric functions, or splines of some order $s$, satisfying
  the conditions of \citep{BarronS91}.  Let $r\geq 3$
  in the polynomial case, $r \geq 2$ in the trigonometric case, and
  $r=s$, $s \geq 2$ in the spline case. Let $k(n)$ be an arbitrary function such that $k(n) \asymp
  n^{1/(2r +1)}$. Then $\sup_{P^* \in \model^{*(r)}} \risk_n(P^*,
  \m{P}_{k(n)}) < \infty$, and $\sup_{P^* \in \model^{*(r)}} \risk_n(P^*,
  \m{P}_{k(n)}) \asymp n^{-2r /(2r +1)}$.
\end{theorem}
The minimax convergence rate for  the models
$\model^{*(r)}$ is given by $\gmm(n)
\asymp n^{-2 r/(2r+1)}$ \citep{YangB98}. Thus, together with
Lemma~\ref{standardnonparametric}, using the oracle $\oracle(P^*,x^n) :=
n^{1/(2r+1)}$, the theorem implies that $\Pswitch$ achieves the
minimax convergence rate.  We note that the paper \citep{BarronS91}
only establishes convergence of KL divergence in probability when maximum likelihood parameters our used. For our
purposes, we need convergence in expectation, which holds when MAP parameters are used, as shown in Sheu's
thesis \citep{Sheu90}. Since the prediction strategies $\m{P}_k$ are based on MAP
estimators rather than on Bayes predictive distributions,
our consistency result Theorem~\ref{thm:consistencyagain} of
Section~\ref{sec:consistency} does not apply. However, by Theorem 2.1 and 2.2 of \citep{LiY00}, we can
apply the alternative consistency result
Theorem~\ref{thm:consistencyb}. Thus, just as for histogram density
estimation as discussed in Section~\ref{sec:histogram}, we do have a
proof of both consistency and minimax rate of convergence for general
nonparametric density estimation with exponential families.

\subsection{Example: Nonparametric Linear Regression}
\label{sec:linear}
\subsubsection{Lemma for Plug-In Estimators}
We first need a variation of Lemma~\ref{standardnonparametric} for the
case that the $\m{P}_k$ are plug-in strategies. We will then apply the
lemma to nonparametric linear regression with $\m{P}_k$ based on
maximum likelihood estimators within $\model_k$.  To prepare for this,
it is useful to rename the observations to $Z_i$ rather than $X_i$.

As before, we assume that $Z_1, Z_2, \ldots$ are i.i.d.\ according to
all $P^* \in \model^*$ and $P \in \model$. We write $D(P^* \| P)$ for
the KL divergence between $P^*$ and $P$ on a single outcome,
i.e.\ $D(P^* \| P) := D(P^*(Z_1 = \cdot) \| P(Z_1 = \cdot))$.  For
given $P^*$, let, if it exists, $\tilde{P}_k$ be the unique $P \in
\model_k$ achieving $\min_{P \in \model_k} D(P^* \| P)$.
\begin{lemma}
\label{lem:linreg}
Let $\model^*$ be a standard nonparametric model class, and let $\m{P}_1, \m{P}_2, \ldots$
be plug-in strategies, i.e.\ for all $k$, all $n, z^n \in
\samplespace^n$, $\m{P}_k(Z_{n+1} = \cdot \mid z^n) \in \{ P(Z_1 =
\cdot ) \mid P \in \model_k\}$. Suppose
that  $\model_1, \model_2, \ldots$ are such that
\begin{enumerate}
\item $\tilde{P}_1, \tilde{P_2}, \ldots$ all exist, 
\item For all $n \geq 1, k \geq 1$, 
$\sup_{P^* \in \model^*} \risk_n(P^* , \m{P}_k) < \infty$, and
\item There exists an oracle $\oracle$ which achieves the minimax rate,
  i.e.\ $\sup_{P^* \in \model^*} \risk_n(P^* \|{P}_{\oracle}) \preceq
  \gmm(n)$, such that $\oracle$ does not look at the data (in the sense
  of Section~\ref{sec:oracle}) and $\oracle(P^*,n-1) \leq n$ for any
  $P^* \in \model^*$ and $n \in \posints$.
\item Furthermore, define the {\em estimation error\/} $\esterr_n(P^*,\m{P}_k) 
:= \risk_n(P^* ,\m{P}_k) - D(P^* \|
  \tilde{P}_k)$, and suppose that for all $k \geq  1$, all $n > k$, 
\begin{equation}
\label{eq:estimerror}
\esterr_{n-1}(P^*,\m{P}_k) \geq \esterr_n(P^*,\m{P}_k).
\end{equation}
\end{enumerate}
Then $\Pswitch$ achieves the minimax rate in Ces\`aro mean, i.e.\ $\Gsw(n) \preceq \Gmmfix(n)$.
\end{lemma}
\begin{proof}
For arbitrary $P^* \in \model^*$ and fixed $\alpha > 1$, let $\oracle'$ be any oracle
that does not depend on the data and that 
``lags a little behind $\oracle$ by at most a factor of $\alpha$'' in the sense of
Lemma~\ref{standardnonparametric}. 
For $n$ such that $n/ \alpha > 1$, let $1 \leq n' \leq n$ be such that $\oracle(P^*,n')  = \oracle'(P^*,n)$. Then
\begin{eqnarray}
\label{eq:droger}
\risk_n(P^*, \m{P}_{\oracle'}) & =  &  D(P^* \| \tilde{P}_{\oracle'})  + 
\esterr_n(P^*, \m{P}_{\oracle'}) 
\nonumber \\
& =&
D(P^* \| \tilde{P}_{\oracle(P^*,n') }) + 
\esterr_n(P^*, \m{P}_{\oracle(P^*,n')}) 
\nonumber \\
& \leq  & D(P^* \| \tilde{P}_{\oracle(P^*,n')} ) + \esterr_{n'}(P^*, \m{P}_{\oracle(P^*,n')}) 
 \nonumber \\
& \preceq & \gmm(n') \preceq (n/\alpha)^{- \gamma} h_0(n) \preceq \gmm(n).
\end{eqnarray}
Here $\esterr_n(P^*, \m{P}_{\oracle(P^*,n')})$ denotes the estimation error when, at sample size $n$, the strategy $\m{P}_k$ with $k = \oracle(P^*, n')$ is used. The last line follows because, by definition of standard nonparametric, $h_0$ is increasing.
For $n$ such that $n/\alpha > 1$, \eqref{eq:droger} in combination
with condition~2 of the lemma (for smaller $n$) shows that we can apply Lemma~\ref{standardnonparametric}, and then the result follows.
\end{proof}
We call $\esterr_n(P^*,\m{P}_k)$ ``estimation error'' since it can be
rewritten as the expected additional logarithmic loss incurred when predicting
$Z_{n+1}$ based on $\m{P}_k$ rather than $\tilde{P}_k$, the best
approximation of $P^*$ within $\model_k$:
$$
\esterr_n(P^*,\m{P}_k) = E_{Z^n \sim P^*} E_{Z_{n+1} \sim P^*} 
[- \log \m{p}_k(Z_{n+1} \mid Z^n) - (- \log \tilde{p}_k(Z_{n+1})) ].
$$ 
As can be seen in the proof
of Lemma~\ref{lem:linregb} below, in the linear regression case,
$\esterr_n(P^*,\m{P}_k)$ can be rewritten as the variance of the
estimator $\m{P}_k$, and thus coincides with the traditional
definition of estimation error.

In order to apply Lemma~\ref{lem:linreg}, one
needs to find an oracle that does not look at the data. A good
candidate to check is the oracle 
\begin{equation}
\label{eq:bestoracle}
\oracle^*(P^*,n) = \arg \min_k \risk_n(P^*,\m{P}_k)
\end{equation}
because, as is immediately verified, if there exists an oracle $\oracle$
that does not look at the data and achieves the minimax rate, then
$\oracle^*$ must achieve the minimax rate as well.
\subsubsection{Nonparametric Linear Regression}
We now apply Lemma~\ref{lem:linreg} to linear regression with random
i.i.d.\ design and i.i.d.\ normally distributed noise with known
variance $\sigma^2$, using least-squares or, equivalently, maximum
likelihood estimators (see Section 6.2 of \citep{Yang99} and Section~4
of \citep{YangB98}). The results below show that $\Pswitch$ achieves
the minimax rate in nonparametric regression under a condition on the
design distribution which we suspect to hold quite generally, but
which is hard to verify. Therefore, unfortunately, our result has
formal implications only for the restricted set of distributions for
which the condition has been verified. We give examples of such sets
below.

Formally, we fix a sequence $\phi_1, \phi_2, \phi_3, \ldots$ of
uniformly bounded, linearly independent functions from $\reals$ to
$\reals$. Let $S_k$ be the space of functions spanned by $\phi_1,
\ldots, \phi_k$.  The linear models $\model_k$ are families of
conditional distributions $P_{\theta}$ for $Y_i \in \reals$ given $X_i
\in {\cal X}$, where ${\cal X} = [0,1]^d$ for some $d > 0$. Here
$\theta = (\theta_1, \ldots, \theta_k) \in \reals^k$ and $P_{\theta}$
expresses that $Y_i = \sum_{j=1}^k \theta_j \phi_j(X_i)+ U_i$, where
the noise random variables $U_1, U_2, \ldots$ are i.i.d.  normally
distributed with zero mean and fixed variance $\sigma^2$. The
prediction strategies $P_1, P_2, \ldots$ are based on maximum
likelihood estimators.  Thus, for $k \leq n$, $\m{P}_k(Y_{n+1} \mid
x^{n+1},y^n) := P_{\hat{\theta}(x^n,y^n)}(Y_{n+1}\mid
X_{n+1}=x_{n+1})$ where $\hat{\theta}(x^n,y^n) \in \reals^k$ and
$P_{\hat{\theta}(x^n,y^n)}$ is the ML estimator within $\model_k$. For
$k > n$, we may set $\m{P}_k(Y_{n+1} \mid x^{n+1},y^n)$ to any fixed
distribution $Q$ with $\sup_{P^* \in \model^*} D(P^*(Y_{n+1} \mid
x_{n+1}) \| Q(Y_{n+1} \mid x_{n+1})) < \infty$.  We denote by
$\Phi^{(n,k)}$ the $k \times n$ design matrix with the $(j,i)-$th
entry given by $\phi_j(x_i)$.

We fix a set of candidate design distribution ${\cal P}^*_X$ and a set
of candidate regression functions ${\cal F}^*$, and we let $\model^*$
denote the set of distributions on $(X_1,Y_1), (X_2, Y_2), \ldots$
such that $X_i$ are i.i.d.\ according to some $P^*_X \in {\cal P}^*_X$,
and $Y_i = f^*(X_i) + U_i$ for some $f^* \in
{\cal F}^*$ and $U_1, U_2, \ldots$ are i.i.d.\ normally distributed
with zero mean and variance $\sigma^2$. We assume that all $f^* \in {\cal F}^*$ can be expressed as 
\begin{equation}
\label{eq:trein}
f^* = \sum_{j=1}^\infty \tilde{\theta}_j \phi_j
\end{equation}
for some $\tilde{\theta}_1, \tilde{\theta}_2, \ldots$ with $\lim_{j
  \rightarrow \infty} \tilde{\theta}_j = 0$.  It is immediate that for such combinations of $\model^*$ and $\model$, condition 1 and 2  of
Lemma~\ref{lem:linreg} hold. The following lemma shows that also condition 4 holds, and thus, if we can also verify that
condition 3 holds, then $\Pswitch$ achieves the minimax rate.
\begin{lemma}
\label{lem:linregb}
Suppose that $\model_1, \model_2, \ldots$ are as above. Let
$\model^*$ be as above, such that additionally, for all $P^* \in
\model^*$, all $n$, all $k \in \{1, \ldots, n\}$, the Fisher information
matrix $(\Phi^{(n,k)})^\transpose (\Phi^{(n,k)})$ is almost surely
nonsingular. Then 
(\ref{eq:estimerror}) holds.
\end{lemma}
A sufficient condition for the required nonsingularity of
$(\Phi^{(n,k)})^\transpose (\Phi^{(n,k)})$ is, for example, that for
all $P^* \in \model^*$, the marginal distribution of $X$ under $P^*$
has a density under Lebesgue measure.
If the conditions of Lemma~\ref{lem:linregb} hold and, additionally, we can
show that some oracle achieves the minimax rate, then condition 3
of Lemma~\ref{lem:linreg} is verified and $\Pswitch$
achieves the minimax rate as well. To verify whether this is the case,
note that
\begin{proposition}
\label{prop:helsinki}
  Suppose that (a) for some $\alpha > 0$, $\sup_{P^*} D(P^* \|
  \tilde{P}_k) \asymp' k^{- 2\alpha}$; (b) $\gmm(n) \asymp' n^{-2
    \alpha / (2 \alpha + 1)}$; and (c) for some $\tau$ with
  $1/(2\alpha +1) \leq \tau < 1$, we have $\esterr_n(P^*,\m{P}_k)
  \asymp' k/n$, uniformly for $k \in \{1,\ldots, n^\tau\}$. Then
  letting, for all $P^* \in \model^*$, $\oracle(P^*,n) := \lceil n^{1
    /(2 \alpha +1)} \rceil$, we have $\sup_{P^* \in \model^*} \risk_n(P^*,
  P_{\oracle}) \preceq' \gmm(n)$.
\end{proposition}
Here $a(n) \preceq' b(n)$ means ``$a(n) \preceq b(n)$ and 
for all $n$, $a(n)$ is finite''. $\asymp'$ is defined in the same way.  
We omit the straightforward proof of Proposition~\ref{prop:helsinki}. Conditions (a) and
(b) hold for many natural combinations of $\model^*$ and $\model$,
under quite weak conditions on $P^*_X$ \citep{Stone82}. Possible
$\model^*$ include regression functions $f^*$ taken from Besov spaces
and Sobolev spaces, and more generally cases where the $\phi_j$ are
`full approximation sets of functions'' (which can be, e.g.,
polynomials, or trigonometric functions) \cite[Section 4]{YangB98}.
\citep{Cox88} shows that also (c) holds under some conditions, but
these are relatively strong; e.g.\ it holds if $P_X$ is a
beta-distribution and $\alpha =1$. We suspect that (c) holds in much
more generality, but we have found no theorem that actually states
this. Note that (c) in fact does hold, even with $\asymp$ replaced by
$=$, if, after having observed $x^n, y^n$, we evaluate $\m{P}_k$ on a
new $X_{n+1}$-value which is chosen uniformly at random from $x_1,
\ldots, x_n$ \citep{Yang99}. But this is of no use to us, since all our
proofs are ultimately based on the connection
(\ref{eq:redundancyrisk}) between the cumulative risk and the KL
divergence. While this connection does not require data to be i.i.d.,
it does break down if we evaluate $\m{P}_k$ on an $X_{n+1}$-value that
is not equal to the value of $X_{n+1}$ that will actually be observed
in the case that additional data are sampled from $P^*$. Therefore, we
cannot extend our results to deal with this alternative evaluation for
which (c) holds automatically. All in all, we can show that the switch
distribution achieves the minimax rate in certain special cases (e.g.
when the conditions of \citep{Cox88} hold for $P^*_X$), but we conjecture
that it holds in much more generality.

\subsection{The Parametric Case}\label{sec:parametric}
We end our treatment of convergence rates by considering the parametric case. 
Thus, in this subsection we assume that $P^*\in\model_{k^*}$ for some
$k^*\in\posints$, but we also consider that if
$\model_1,\model_2,\ldots$ are of increasing complexity, then the
catch-up phenomenon may occur, meaning that at small sample sizes,
some estimator $\Pbayes_k$ with $k<k^*$ may achieve smaller risk than
$\Pbayes_{k^*}$. In particular, this can happen if $P^* \in
\model_{k^*}, P^* \not \in \model_{k^*-1}$, but $D(P^* \|
\model_{k^*-1}) := \inf_{P \in \model_{k^* -1}}$ is small.
\Citet{vanerven2006} shows that in some scenarios, there exist
i.i.d.\ sequences $X_1$, $X_2$, $\ldots$ with $P^*(X_i) \in
\model_{k^*}$ for all $i \in \posints$, such that $\lim_{m \rightarrow
\infty} D(P^*_{(m)} \| \model_{k^*-1}) = 0$ and $\lim_{m \rightarrow
\infty} \lim_{n \rightarrow \infty} \crisk_n(P^*_{(m)} \| \Pbma) -
\crisk_n(P^*_{(m)} , \Pswitch) = \infty$. That is, the difference in
cumulative risk between $\Pswitch$ and $\Pbma$ may become arbitrarily
large if $D(P^*_{(m)} \| \model_{k^*-1})$ is chosen small enough. Thus,
even in the parametric case $\Pbma$ is not always optimal: if $P^* \in
\model_{k^*}$, then, as soon as we also put a positive prior weight on
$\m{P}_{k^*-1}$, $\Pbma$ may favour $k^*-1$ at sample sizes at which
$\m{P}_{k^*}$ has already become the best predictor.  The following
lemma shows that in such cases the switch distribution remains optimal:
the predictive performance of the switch distribution is never much
worse than the predictive performance of the best oracle that iterates
through the models in order of increasing complexity. In order to extend
this result to a formal proof that $\Pswitch$ always achieves the
minimax convergence rate, we would have to additionally show that there
exist oracles of this kind that achieve the minimax convergence rate.
Although we have no formal proof of this extension, it seems likely that
this is the case.
\begin{lemma}\label{lem:parametric}
  Let $\Pswitch$ be the switch distribution, defined with respect to a
  sequence of estimators $\m{P}_1,\m{P}_2,\ldots$ as above, with prior
  $\switchprior$ satisfying~\eqref{eq:mmprior}. Let $k^*\in\posints$,
  and let $\oracle$ be any oracle such that for any $P^* \in \model^*$,
  any $x^\infty \in \samplespace^\infty$, the sequence $\oracle_1$,
  $\oracle_2$, $\ldots$ is nondecreasing and there exists some $n_0$ such
  that $\oracle_n = k^*$ for all $n \geq n_0$, where $\oracle_i \equiv
  \oracle(P^*,x^{i-1})$ for all $i$. Then
  \begin{equation}
    \Gsw(n) - G_\oracle(n)
      \leq \sup_{P^*\in\model^*}\left(
        \sum_{i=1}^n \risk_i(P^*,\Pswitch)-\sum_{i=1}^n \risk_i(P^*,P_\oracle)\right)
      =k^*\cdot O(\log n).
    \label{eqn:parametric}
  \end{equation}
  Consequently, if $G_\oracle(n) \succeq \log n$,
  then
  \begin{equation}
    \Gsw(n) \preceq G_\oracle(n).
    \label{eqn:parametric2}
  \end{equation}
\end{lemma}
%
%
\begin{proof}
  The inequality in \eqref{eqn:parametric} is a consequence of the
  general fact that $\sup_x f(x)-\sup_x f'(x)\leq \sup_x (f(x)-f'(x))$
  for any two functions $f$ and $f'$. The second
  part of \eqref{eqn:parametric} follows by
  Lemma~\ref{lem:simulateoracle}, applied with $g(n) = k^*$, together
  with the observation that $m_\oracle(n) \leq k^*$. To show
  \eqref{eqn:parametric2} we can apply Lemma~\ref{lem:rateofconvergence}
  with $g(n) = k^*$ and $f(n) = G_\oracle(n)$. (Condition~\ref{cond:iii}
  of the lemma is satisfied with $c_1 = 1$, and by assumption about
  $G_\oracle(n)$ there exists a constant $c_2$ such that
  Condition~\ref{cond:ii} of the lemma is satisfied.)
\end{proof}

The lemma shows that the additional cumulative risk of the
switch distribution compared to $P_\oracle$ is of order $\log n$. In the
parametric case, we usually have $\Gmmfix(n)$ proportional to $\log n$
(Section~\ref{sec:standard}). If that is the case, and if, as seems
reasonable, there is an oracle $\oracle$ that satisfies the given
restrictions and that achieves summed risk proportional to $\Gmmfix(n)$,
then also the switch distribution achieves a summed risk that is
proportional to $\Gmmfix(n)$.
\section{Efficient Computation of the switch distribution}
\label{sec:computation}
\newcommand{\allowedk}{\mathcal K}
\newcommand{\kmax}{K_\text{max}}

For priors $\pi$ as in~\eqref{eq:prior}, the posterior probability on
predictors $p_1,p_2,\ldots$ can be efficiently computed sequentially,
provided that $\pit(T=n\mid T\geq n)$ and $\pik$ can be calculated
quickly (say in constant time) and that $\pim(m)=\theta^m(1-\theta)$ is
geometric with parameter $\theta$, as is also required for
Theorem~\ref{thm:consistencyagain} and (see
Section~\ref{sec:restrictions})  permitted in the theorems and lemma's
of Section~\ref{sec:nonparametrica} and~\ref{sec:nonparametricb}. For
example, we may take $\pik(k) = 1/(k(k+1))$ and $\pit(n)
= 1/(n(n+1))$, such that $\pit(T = n \mid T \geq n) = 1/n$.

The algorithm resembles \textsc{Fixed-Share}~\citep{HerbsterWarmuth1998},
but whereas \textsc{Fixed-Share} implicitly imposes a geometric
distribution for $\pit$, we allow general priors by varying the shared
weight with $n$. We also add the $\pim$ component of the prior, which is
crucial for consistency. This addition ensures that the additional loss
compared to the best prediction strategy that switches a finite number
of times, does not grow with the sample size.

To ensure finite running time, we need to restrict the
switch distribution to switch between a finite number of prediction
strategies. This is no strong restriction though, as we may just take
the number of prediction strategies sufficiently large relative to $N$
when computing $\pswitch(x^N)$. For example, consider the
switch distribution that switches between prediction strategies $p_1$,
$\ldots$, $p_{\kmax(N)}$. Then all the theorems in the paper can still
be proved if we take $\kmax(N)$ sufficiently large (e.g.\ $\kmax(N) \geq
g(N)$ would suffice for the oracle approximation lemma).

This is a special case of a switch distribution that, at sample size
$n$, allows switching only to $p_k$ such that $k \in \allowedk_n
\subseteq \posints$, where $\allowedk_1 \subseteq \allowedk_2 \subseteq
\ldots$. We may view this as a restriction on the prior:
$\pi(\switchpars \setminus \switchpars') = 0$, where
\begin{equation}
  \switchpars':=\{\switchpar\in\switchpars
    \mid\forall n \in\posints:K_n(\switchpar)\in\allowedk_n\}
\end{equation}
denotes the set of allowed parameters, and, as in
Section~\ref{sec:switchdefinition},
\begin{equation}
  K_n(\switchpar) := k_i(\switchpar)\textnormal{ for the unique $i$ such that
  $t_i(\switchpar)<n$ and $i=m(\switchpar)\vee t_{i+1}(\switchpar)\ge n$}
\end{equation}
denotes which prediction strategy is used to predict outcome $X_n$.

The following online algorithm computes the switch distribution
for any $\allowedk_1 \subseteq \allowedk_2 \subseteq \ldots$, provided
the prior is of the form \eqref{eq:prior}. Let the indicator function,
$\ind_A(x)$, be $1$ if $x\in A$ and $0$ otherwise.


\medskip\noindent\begin{minipage}{\textwidth}
\begin{algorithm}{\textsc{Switch}$(x^N)$}\label{algo:switch}
  \\ \For $k\in\mathcal K_1$ \Do initialize $w_k^a~\=~\pik(k)\cdot\theta$;
  $w_k^b~\=~\pik(k)\cdot(1-\theta)$ \textbf{od}\label{line:init}
  \\ \For $n\!=\!1,\ldots,N$ \Do\label{line:iter}
  \>
  \\ Report $\pswitch(K_n,x^{n-1})=w^a_{K_n}\!\!+w^b_{K_n}$\quad(a $K$-sized array)\label{line:posterior}
  \\ \For $k\in\mathcal K_n$ \Do $w_k^a~\=~w_k^a\cdot
  p_k(x_n|x^{n-1})$; $w_k^b~\=~w_k^b\cdot p_k(x_n|x^{n-1})$ \textbf{od}\hfill\textit{(loss update)}\label{line:loss_update}
  \\ \texttt{pool}~\=~$\pit(Z=n\mid Z\ge n)\cdot \sum_{k\in\mathcal K_n}
  w_k^a$\label{line:pool}
  \\ \For $k\in\mathcal K_{n+1}$ \Do\label{line:share_update_start}
  \> 
  \\ $w_k^a ~\=~ w_k^a\cdot \ind_{\mathcal K_n}(k)\cdot\pit(Z\ne n\mid Z\ge n)~+~\texttt{pool}\cdot\pik(k)\cdot\theta$\hfill\textit{(share update)}
  \\ $w_k^b ~\=~ w_k^b\cdot \ind_{\mathcal K_n}(k)~+~\texttt{pool}\cdot\pik(k)\cdot(1-\theta)$
  \<
  \\ \textbf{od}\label{line:share_update_end}
  \<
  \\ \textbf{od}
  \\ Report $\pswitch(K_{N+1},x^N)=w^a_{K_{N+1}}\!\!+w^b_{K_{N+1}}$\label{line:last_posterior}

\end{algorithm}
\end{minipage}

\bigskip\noindent This algorithm can be used to obtain fast convergence
in the sense of Sections~\ref{sec:nonparametrica}
and~\ref{sec:nonparametricb}, and consistency in the sense of
Theorem~\ref{thm:consistencyagain}.
If $\pit(T = n \mid T \geq n)$ and $\pik$ can be computed in constant
time, then its running time is $\Theta(\sum_{n=1}^N|\mathcal K_n|)$,
which is typically of the same order as that of fast model selection
criteria like AIC and BIC. For example, if the number of considered
prediction strategies is fixed at $\kmax$ then the running time is
$\Theta(\kmax\cdot N)$.

\begin{theorem}\label{thm:algo}
  Let $\pswitch$ denote the switch distribution with prior $\pi$.
  Suppose that $\pi$ is of the form \eqref{eq:prior} and
  $\pi(\switchpars \setminus \switchpars') = 0$. Then
  Algorithm~\ref{algo:switch} correctly reports $\pswitch(K_1,x^0)$,
  $\ldots$, $\pswitch(K_{N+1},x^N)$.
\end{theorem}

Note that the posterior $\pi(K_{N+1} \mid x^N)$ and the marginal likelihood
$\pswitch(x^N)$ can both be computed from $\pswitch(K_{N+1},x^N)$ in
$\Theta(|\allowedk_{N+1}|)$ time. The theorem is proved in
Appendix~\ref{sec:proofalgo}.

\section{Relevance and Earlier Work}\label{sec:relevance} 
Over the last 25 years or so, the question whether to base model
selection on AIC or BIC type methods has received a lot of attention
in the theoretical and applied statistics literature, as well as in
fields such as psychology and biology where model selection plays an
important role (googling ``AIC'' \emph{and} ``BIC'' gives 355000 hits)
\citep{SpeedY93,HansenY01,HansenY02,BarronYY94,Forster01,DeLunaS03,Sober04}.
It has even been suggested that, since these two types of methods have
been designed with different goals in mind (optimal prediction vs.
``truth hunting''), it may simply be the case that {\em no\/} procedures exist
that combine the best of both types of approaches \citep{Sober04}. Our
Theorem~\ref{thm:consistencyagain}, Theorem~\ref{thm:switchrsy} and
our results in Section~\ref{sec:nonparametricb} show that, at least in
some cases, one can get the best of both worlds after all, and model
averaging based on $\Pswitch$ achieves the minimax optimal convergence
rate. In typical parametric settings $(P^* \in \model)$, model
selection based on $\Pswitch$ is consistent, and
Lemma~\ref{lem:parametric} suggests that model averaging based on
$\Pswitch$ is within a constant factor of the minimax optimal rate in
parametric settings.
\subsection{A Contradiction with Yang's Result?}
\label{sec:aicbic}
Superficially, our
results may seem to contradict the central conclusion of Yang
\citep{Yang05a}.  Yang shows that there are scenarios in linear
regression where no model selection or model combination criterion can
be both consistent and achieve the minimax rate of convergence.

Yang's result is proved for a variation of linear regression in which
the estimation error is measured on the previously observed design
points. This setup cannot be directly embedded in our framework. Also,
Yang's notion of model combination is somewhat different from the
model averaging that is used to compute $\Pswitch$. Thus, formally,
there is no contradiction between Yang's results and ours. Still, the
setups are so similar that one can easily imagine a variation of
Yang's result to hold in our setting as well. Thus, it is useful to
analyze how these ``almost'' contradictory results may coexist. We
suspect (but have no proof) that the underlying reason is the
definition of our minimax convergence rate in Ces\`aro mean
(\ref{eq:mmfixed}) in which $P^*$ is allowed to depend on $n$, but
then the risk with respect to that same $P^*$ is summed over all
$i=1,\ldots,n$. In contrast, Yang uses the standard definition of
convergence rate, without summation. Yang's result holds in a
parametric scenario, where there are two nested parametric models, and
data are sampled from a distribution in one of them. Then both
$\Gmmfix$ and $\Gmmvar$ are of the same order $\log n$. Even so, it
may be possible that there does exist a minimax optimal procedure that
is also consistent, relative to the $\Gmmfix$-game, in which $P^*$ is
kept fixed once $n$ has been determined, while there does not exist a
minimax optimal procedure that is also consistent, relative to the
$\Gmmvar$-game, in which $P^*$ is allowed to vary. We conjecture that
this explains why Yang's result and ours can coexist: in {\em
  parametric\/} situations, there exist procedures (such as
$\Pswitch$) that are both consistent and achieve $\Gmmfix$, but there
exist no procedures that are both consistent and achieve $\Gmmvar$. We
suspect that the qualification ``parametric'' is essential here:
indeed, we conjecture that in the standard {\em nonparametric\/} case,
whenever $\Pswitch$ achieves the fixed-$P^*$ minimax rate $\Gmmfix$, it
also achieves the varying-$P^*$ minimax rate $\Gmmvar$. The reason for
this conjecture is that, under the standard nonparametric assumption,  whenever $\Pswitch$ achieves
$\Gmmfix$, a small modification of $\Pswitch$ will achieve $\Gmmvar$.
Indeed, define the {\em Ces\`aro-switch distribution\/} as
\begin{equation}
  \Pcesaroswitch(x_{n} \mid x^{n-1})
    := \frac{1}{n} \sum_{i=1}^n \Pswitch(x_n \mid x^{i-1}).
\end{equation}
\begin{proposition}
\label{prop:yang}
$\Pcesaroswitch$ achieves the varying-$P^*$-minimax rate
whenever $\Pswitch$ achieves the fixed-$P^*$-minimax rate.
\end{proposition}
The proof of this proposition is similar to the proof of
Proposition~\ref{prop:varfix} and can be found in
Section~\ref{sec:varfix}. 

Since, intuitively, $\Pcesaroswitch$ learns ``slower'' than
$\Pswitch$, we suspect that $\Pswitch$ itself achieves the
varying-$P^*$-minimax rate as well in the standard nonparametric case.
However, while in the nonparametric case, $\gmm(n) \asymp
\Gmmfix(n)/n$, in the parametric case, $\gmm(n) \asymp 1/n$ whereas
$\Gmmfix(n)/n \asymp (\log n)/n$. Then the reasoning underlying
Proposition~\ref{prop:yang} does not apply anymore, and
$\Pcesaroswitch$ may not achieve the minimax rate for varying $P^*$.
Then also $\Pswitch$ itself may not achieve this rate.  We suspect
that this is not a coincidence: Yang's result suggests that indeed, in
this parametric setting, $\Pswitch$, because it is consistent, {\em
  cannot\/} achieve this varying $P^*$-minimax optimal rate.
\subsection{Earlier Approaches to the AIC-BIC Dilemma}
Several other authors have provided procedures which have been
designed to behave like AIC whenever AIC is better, and like BIC
whenever BIC is better; and which empirically seem to do so; these
include {\em model meta-selection\/} \citep{DeLunaS03,Clarke97}, and
Hansen and Yu's {\em gMDL\/} version of MDL regression
\citep{HansenY01}; also the ``mongrel'' procedure of \citep{wong2004}
has been designed to improve on Bayesian model averaging for small
samples. Compared to these other methods, ours seems to be the first
that {\em provably\/} is both consistent and minimax optimal in terms
of risk, for some classes $\model^*$. The only other procedure that we know of for which somewhat
related results have been shown, is a version of cross-validation
proposed by \citet{Yang05b} to select between AIC and BIC in
regression problems.  Yang shows that a particular form of
cross-validation will asymptotically select AIC in case the use of AIC leads to
better predictions, and BIC in the case that BIC leads to better
predictions.  In contrast to Yang, we use a single paradigm
rather than a mix of several ones (such as AIC, BIC and
cross-validation) -- essentially our paradigm is just that of
universal individual-sequence prediction, or equivalently, the
individual-sequence version of predictive MDL, or equivalently,
Dawid's prequential analysis applied to the log scoring rule. Indeed,
our work has been heavily inspired by prequential ideas; in 
\citet{Dawid92b} it is already suggested that model selection should be
based on the {\em transient\/} behaviours in terms of sequential
prediction of the estimators within the models: one should select the
model which is optimal at the given sample size, and this will change
over time.  Although Dawid uses standard Bayesian mixtures of
parametric models as his running examples, he implicitly suggests that
other ways (the details of which are left unspecified) of combining
predictive distributions relative to parametric models may be
preferable, especially in the nonparametric case where the true
distribution is outside any of the parametric models under
consideration.
\subsection{Prediction with Expert Advice}
\label{sec:predictionexpert}
Since the switch distribution has been designed to perform well in a
setting where the optimal predictor $\bar{p}_k$ changes over time, our
work is also closely related to the algorithms for {\em tracking the
  best expert\/} in the universal prediction literature
\citep{HerbsterWarmuth1998,Vovk1999,volfwillems1998,MonteleoniJ04}.
However, those algorithms are usually intended for data that are
sequentially generated by a mechanism whose behaviour changes over
time. In sharp contrast, our switch distribution is especially
suitable for situations where data are sampled from a {\em fixed\/}
(though perhaps non-i.i.d.) source after all; the fact that one model
temporarily leads to better predictions than another is caused by the
fact that each ``expert'' $\m{p}_k$ has itself already been designed
as a universal predictor/estimator relative to some large set of
distributions $\model_k$. The elements of $\model_k$ may be viewed as
``base'' predictors/experts, and the $\m{p}_k$ may be thought of as
meta-experts/predictors. Because of this two-stage structure, which
meta-predictor $\m{p}_k$ is best changes over time, even though the
optimal base-predictor $\argmin_{p\in\model}\risk_n(p^*,p)$ does not
change over time.

If one of the considered prediction strategies $\m{p}_k$ makes the
best predictions eventually, our goal is to achieve consistent model
selection: the total number of switches should also remain bounded. To
this end we have defined the switch distribution such that positive
prior probability is associated with switching finitely often and
thereafter using $\m{p}_k$ for all further outcomes. We need this
property to prove that our method is consistent. Other dynamic expert
tracking algorithms, such as the \textsc{Fixed-Share} algorithm
\citep{HerbsterWarmuth1998}, have been designed with different goals in
mind, and as such they do not have this property. Not surprisingly
then, our results do not resemble any of the existing results in the
``tracking''-literature.

\section{The Catch-Up
  Phenomenon, Bayes and Cross-Validation}\label{sec:ketchup}
\subsection{The Catch-Up Phenomenon is Unbelievable! (According to BMA)} 
\label{sec:unbelievabletruth}
On page~\pageref{eq:bayesmargintro}
we introduced the marginal Bayesian 
distribution $\pbma(x^n) := \sum_k w(k) \m{p}_k(x^n)$.  If the
distributions $\m{p}_k$ are themselves Bayesian marginal distributions
as in~\eqref{eq:bayesint}, then $\pbma$ may be interpreted as (the
density corresponding to) a distribution on the data that reflects
some prior beliefs about the domain that is being modelled, as
represented by the priors $w(k)$ and $w_k(\theta)$.  If $w(k)$ and
$w_k(\theta)$ truly reflected some decision-maker's a priori beliefs,
then it is clear that the decision-maker would like to make sequential
predictions of $X_{n+1}$ given $X^n = x^n$ based on $\pbma$ rather
than on $\pswitch$. Indeed, as we now show, the catch-up phenomenon as
depicted in Figure~\ref{fig:markovexample} is exceedingly unlikely to
take place under $\pbma$, and {\em a priori\/} a subjective Bayesian
should be prepared to bet a lot of money that it does not occur. To see
this, consider the {\em no-hypercompression inequality\/}
\citep{grunwald2007}, versions of which are also known as ``Barron's
inequality'' \citep{Barron85} and ``competitive optimality of the
Shannon-Fano code'' \citep{cover1991}. It states that for any two
distributions $P$ and $Q$ for $X^{\infty}$, the $P$-probability that
$Q$ outperforms $P$ by $k$ bits or more when sequentially predicting
$X_1, X_2, \ldots$ is exponentially small in $k$: for each $n$,
$$
P(- \log q(X^n) \leq - \log p(X^n) - k ) \leq 2^{-k}.
$$
Plugging in $\pbma$ for $p$, and $\pswitch$ for $q$, we see that what
happened in Figure~\ref{fig:markovexample} ($\pswitch$ outperforming
$\pbma$ by about 40000 bits) is an event with probability no more than
$2^{-40000}$ according to $\pbma$. Yet, in many practical situations,
the catch-up phenomenon does occur and $\pswitch$ gains significantly
compared to $\pbma$.  This can only be possible because either the
models are wrong (clearly, The Picture of Dorian Gray has not been
drawn randomly from a finite-order Markov chain), or because the
priors are ``wrong'' in the sense that they somehow don't match the
situation one is trying to model. For this reason, some subjective
Bayesians, when we confronted them with the catch-up phenomenon, have
argued that it is just a case of ``garbage in, garbage out'' (GIGO):  when the phenomenon occurs, then, rather than using the
switch distribution, one should reconsider the model(s) and prior(s)
one wants to use, and, once one has found a superior model $\model'$
and prior $w'$, one should use $\pbma$ relative to $\model'$ and
$w'$. Of course we agree that {\em if\/} one can come up with better
models, one should of course use them. Nevertheless, we strongly  
disagree with the GIGO point of view: We are convinced that in
practice, ``correct'' priors may be impossible to obtain; similarly,
people are forced to work with ``wrong'' models all the time. In such
cases, rather than embarking on a potentially never-ending quest for
better models, the hurried practitioner may often prefer to use the
imperfect -- yet still useful -- models that he has available, {\em in
  the best possible manner}. And then it makes sense to use $\pswitch$
rather than the Bayesian $\pbma$: the best
one can hope for in general is to regard the distributions in one's
models as prediction strategies, and try to predict as well as the
best strategy contained in any of the models, and $\pswitch$ is better
at this than $\pbma$. Indeed, the catch-up phenomenon raises
some interesting questions for Bayes factor model selection: no matter
what the prior is, by the no-hypercompression inequality above with $p
= \pbma$ and $q = \pswitch$, when comparing two models $\model_1$ and
$\model_2$, before seeing any data, a Bayesian {\em always\/} believes
that the switch distribution will not substantially outperform
$\pbma$, which implies that a Bayesian {\em cannot\/} believe that,
with non-negligible probability, a complex model $\m{p}_2$ can at
first predict substantially worse than a simple model $\m{p}_1$ and
then, for large samples, can predict substantially better. Yet in
practice, this happens all the time!

\subsection{Nonparametric Bayes}
A more interesting subjective Bayesian argument against the switch
distribution would be that, in the nonparametric setting, the data are
sampled from some $P^* \in \model^*\setminus\model$, and is not
contained in any of the parametric models $\model_1, \model_2, \ldots$
Yet, under the standard hierarchical prior used in $\pbma$ (first a
discrete prior on the model index, then a density on the model
parameters), we have that with prior-probability 1, $P^*$ is
``parametric'', i.e.\ $P^* \in \model_k$ for some $k$. Thus, our prior
distribution is not really suitable for the situation that we are
trying to model in the nonparametric setting, and we should use a
nonparametric prior instead. While we completely agree with this
reasoning, we would immediately like to add that the question then
becomes: what nonparametric prior {\em should\/} one use?
Nonparametric Bayes has become very popular in recent years, and it
often works surprisingly well. Still, its practical and theoretical
performance strongly depends on the type of priors that are used, and
it is often far from clear what prior to use in what situation. In
some situations, some nonparametric priors achieve optimal rates of
convergence, but others can even make Bayes inconsistent
\citep{DiaconisF86,grunwald2007}. The advantage of the
switch distribution is that it does not require any difficult modeling
decisions, but nevertheless under reasonable conditions it achieves
the optimal rate of convergence in nonparametric settings, and, in the
special case where one of the models on the list in fact approximates
the true source extremely well, this model will in fact be identified
(Theorem~\ref{thm:consistencyagain}). In fact, one may think of
$\pswitch$ as specifying a very special kind of nonparametric prior,
and under this interpretation, our results are in complete agreement
with the nonparametric Bayesian view.

\subsection{Leave-One-Out Cross-Validation} From the other side of the
spectrum, it has sometimes been argued that consistency is irrelevant,
since in practical situations, the true distribution is never in any
of the models under consideration. Thus, it is argued, one should use
AIC-type methods such as leave-one-out cross-validation, because of
their predictive optimality. We strongly disagree with this argument,
for several reasons: first, in practical model selection problems, one
is often interested in questions such as ``does $Y$ depend on feature
$X_k$ or not?'' For example, $\model_{k-1}$ is a set of conditional
distributions in which $Y$ is independent of $X_k$, and $\model_{k}$
is a superset thereof in which $Y$ can be dependent on $X_k$.  There
are certainly real-life situations where some variable $X_j$ is truly
completely irrelevant for predicting $Y$, and it may be the primary
goal of the scientist to find out whether or not this is the case. In
such cases, we would hope our model selection criterion to select, for
large $n$, $\model_{k-1}$ rather than $\model_{k}$, and the problem
with the AIC-type methods is that, because of their inconsistency,
they sometimes do not do this. In other words, we think that
consistency does matter, and we regard it as a clear advantage of the
switch distribution that it is consistent.

A second advantage over leave-one-out cross-validation is that the
switch distribution, like Bayesian methods, satisfies Dawid's {\em
  weak prequential principle\/} \citep{Dawid92b,grunwald2007}: the
switch distribution assesses the quality of a predictor $\m{p}_k$ only
in terms of the quality of predictions {\em that were actually made}.
To apply LOO on a sample $x_1, \ldots, x_n$, one needs to know the
prediction for $x_i$ given $x_1, \ldots, x_{i-1}$, but also $x_{i+1},
\ldots, x_n$. In practice, these may be hard to compute, unknown or
even unknowable. An example of the first are non-i.i.d.\ settings such
as time series models. An example of the second is the case where the
$\m{p}_k$ represent, for example, weather forecasters, or other
predictors which have been designed to predict the future given the
past. Actual weather forecasters use computer programs to predict the
probability that it will rain the next day, given a plethora of data
about air pressure, humidity, temperature etc. and the pattern of rain
in the past days. It may simply be impossible to apply those programs
in a way that they predict the probability of rain today, given data
about tomorrow.

\section{Conclusion and Future Work}
We have identified the catch-up phenomenon as the underlying reason
for the slow convergence of Bayesian model selection and averaging.
Based on this, we have defined the switch distribution $\Pswitch$, a
modification of the Bayesian marginal distribution which is
consistent, but also under broad conditions achieves a minimax optimal
convergence rate, thus resolving the AIC-BIC dilemma.
\begin{enumerate}
\item Since $\pswitch$ can be computed in practice, the approach can
  readily be tested with real and simulated data in both density
  estimation and regression problems. Initial results on simulated
  data, on which we will report elsewhere, give empirical evidence
  that $\pswitch$ behaves remarkably well in practice.  Model
  selection based on $\pswitch$, like for $\pbma$, typically
  identifies the true distribution at moderate sample
  sizes. Prediction and estimation based on $\Pswitch$ is of
  comparable quality to leave-one-out cross-validation (LOO) and
  generally, in no experiment did we find that it behaved
  substantially worse than either LOO or AIC.

\item It is an interesting open question whether there is an analogue
  of Lemma~\ref{lem:rateofconvergence} and Theorem~\ref{thm:switchrsy} for
  model {\em selection\/} rather than averaging. In other words, in
  settings such as histogram density estimation where model averaging
  based on the switch distribution achieves the minimax convergence
  rate, does model selection based on the switch distribution achieve
  it as well? For example, in Figure~\ref{fig:markovexample},
  sequentially predicting by the $\m{p}_{K_{n+1}}$ that has maximum a
  posteriori probability (MAP) under the switch distribution given
  data $x^n$, is only a few bits worse than predicting by model
  averaging based on the switch distribution, and still outperforms
  standard Bayesian model averaging by about $40\,000$ bits. In the
  experiments mentioned above, we invariably found that predicting by
  the MAP $\m{p}_{K_{n+1}}$ empirically converges at the same rate as
  using model averaging, i.e.\ predicting by $\Pswitch$.  However, we
  have no proof that this really must always be the case.  Analogous
  results in the MDL literature suggest that a theorem bounding the
  risk of switch-based model selection, if it can be proved at all,
  would bound the squared Hellinger rather than the KL risk
  \citep[Chapter 15]{grunwald2007}.
\item The way we defined $\Pswitch$, it does not seem suitable for
  situations in which the number of considered models or model
  combinations is exponential in the sample size. Because of condition
  (i) in Lemma~\ref{lem:rateofconvergence}, our theoretical results do not
  cover this case either. Yet this case is highly important in
  practice, for example, in the subset selection problem
  \citep{Yang99}. It seems clear that the catch-up phenomenon can and
  will also occur in model selection problems of that type. Can our
  methods be adapted to this situation, while still keeping the
  computational complexity manageable? And what is the relation with the
  popular and computationally efficient $L_1$-approaches to model
  selection \citep{Tibshirani96}?
\end{enumerate}
\section*{Acknowledgements}
We thank Peter Harremo\"es for his crucial help in the proof of
Theorem~\ref{thm:consistencyagain}, and Wouter Koolen for pointing out
a serious error in the proof and interpretation of Theorem 2 of the
preliminary version \citep{ErvenGR07} of (a part of) this paper (This
error had gone unnoticed by the reviewers).  We are very grateful to
Yishay Mansour, who made a single remark over lunch at COLT 2005 that
sparked off all this research, and Andrew Barron for some very helpful
conversations.  This work was supported in part by the IST Programme
of the European Community, under the PASCAL Network of Excellence,
IST-2002-506778. This publication only reflects the authors' views.

\appendix
\section{Proofs}
\label{sec:proofs}
\subsection{Proof  of Theorem~\ref{thm:consistencyagain} }

Let $U_n = \{ \switchpar \in \switchpars \mid K_{n+1}(\switchpar) \neq
k^*\}$ denote the set of ``bad'' parameters $\switchpar$ that
select an incorrect model. It is sufficient to show that
\begin{equation}
  \label{eqn:badcasesdie}
  \lim_{n \to \infty} \frac{\sum_{\switchpar \in U_n}
        \pi\big(\switchpar\big) q_{\switchpar}(X^n)}
      {\sum_{\switchpar \in \switchpars}
        \pi\big(\switchpar\big) q_{\switchpar}(X^n)} 
  = 0
  \quad \quad \text{with $\Pbayes_{k^*}$-probability 1.}
\end{equation}
To see this, first note that \eqref{eqn:badcasesdie} is almost
equivalent to \eqref{eqn:thmconsistency}. The difference is that
$P_{\theta^*}$-probability has been replaced by
$\Pbayes_{k^*}$-probability. Now suppose the theorem is false. Then
there exists a set of parameters $\Phi \subseteq \Theta_{k^*}$ with
$w_{k^*}(\Phi) > 0$ such that \eqref{eqn:thmconsistency} does not hold
for any $\theta^* \in \Phi$. But then by definition of $\Pbayes_{k^*}$
we have a contradiction with \eqref{eqn:badcasesdie}.

To show \eqref{eqn:badcasesdie}, let $A = \{\switchpar\in \switchpars :
k_m(\switchpar) \neq k^*\}$ denote the set of parameters that are bad
for all sufficiently large $n$. We observe that for each $\switchpar'\in
U_n$ there exists at least one element $\switchpar \in A$ that uses the
same sequence of switch-points and predictors on the first $n+1$
outcomes (this implies that $K_i(\switchpar)=K_i(\switchpar')$ for $i =
1, \ldots, n+1$) and has no switch-points beyond $n$ (i.e.\
$t_m(\switchpar) \leq n$). Consequently, either $\switchpar'=\switchpar$
or $\switchpar'\in E_\switchpar$. Therefore
\begin{equation}
  \label{eqn:onlyAmatters}
  \sum_{\switchpar' \in U_n} \pi(\switchpar') q_{\switchpar'}(x^n)
  ~\leq~ \sum_{\switchpar \in A} \left(\pi(\switchpar) +
      \pi(E_{\switchpar})\right) q_{\switchpar}(x^n) 
  ~\leq~ (1+c) \sum_{\switchpar \in A} \pi(\switchpar) q_{\switchpar}(x^n).
\end{equation}
Defining the mixture $r(x^n) = \sum_{\switchpar \in A} \pi(\switchpar)
q_\switchpar(x^n)$, we will show that
\begin{equation}
  \label{eqn:Adies}
  \lim_{n \to \infty} \frac{r(X^n)}
    {\pi(\switchpar=(0,k^*))\cdot \pbayes_{k^*}(X^n)} = 0
  \quad \quad \text{with $\Pbayes_{k^*}$-probability 1.}
\end{equation}
Using \eqref{eqn:onlyAmatters} and the fact that
 $ \sum_{\switchpar \in \switchpars} \pi(\switchpar) q_{\switchpar}(x^n)
    \geq \pi(\switchpar = (0,k^*))\cdot \pbayes_{k^*}(x^n)$,
this implies \eqref{eqn:badcasesdie}. 

For all $\switchpar\in A$ and $x^{t_m(\switchpar)} \in
\samplespace^{t_m(\switchpar)}$, by definition
$Q_{\switchpar}(X_{t_m+1}^\infty|x^{t_m})$ equals
$\Pbayes_{k_m}(X_{t_m+1}^\infty|x^{t_m})$, which is mutually singular
with $\Pbayes_{k^*}(X_{t_m+1}^\infty|x^{t_m})$ by assumption. If
$\samplespace$ is a separable metric space, which holds because
$\samplespace \subseteq \reals^d$ for some $d \in \posints$, it can be
shown that this conditional mutual singularity implies mutual
singularity of $Q_{\switchpar}(X^\infty)$ and
$\Pbayes_{k^*}(X^\infty)$. To see this for countable $\samplespace$,
let $B_{x^{t_m}}$ be any event such that
$Q_{\switchpar}(B_{x^{t_m}}|x^{t_m})=1$ and
$\Pbayes_{k^*}(B_{x^{t_m}}|x^{t_m})=0$. Then, for $B = \{y^\infty \in
\samplespace^\infty \mid
y_{t_m+1}^\infty \in B_{y^{t_m}}\}$, we have that $Q_\switchpar(B) = 1$
and $\Pbayes_{k^*}(B) = 0$. In the uncountable case, however, $B$ may
not be measurable. In that case, the proof follows by Corollary~\ref{cor:harremoes} proved in Section~\ref{sec:harremoes}.
Any countable mixture of distributions that are mutually singular with
$P_{k^*}$, in particular $R$, is mutually singular with $P_{k^*}$. This
implies \eqref{eqn:Adies} by Lemma~3.1 of \citep{Barron85}, which says
that for any two mutually singular distributions $R$ and $P$, the
density ratio $r(X^n)/p(X^n)$ goes to zero as $n \rightarrow \infty$
with $P$-probability $1$.\qed

\subsection{Proof of Theorem~\ref{thm:consistencyb}}

The proof is almost identical to the proof of
Theorem~\ref{thm:consistencyagain}. Let $U_n = \{ \switchpar \in
\switchpars \mid K_{n+1}(\switchpar) \neq k^*\}$ denote the set of
``bad'' parameters $\switchpar$ that select an incorrect model. It is
sufficient to show that
\begin{equation}
  \label{eqn:badcasesdieb}
  \lim_{n \to \infty} \frac{\sum_{\switchpar \in U_n}
        \pi\big(\switchpar\big) q_{\switchpar}(X^n)}
      {\sum_{\switchpar \in \switchpars}
        \pi\big(\switchpar\big) q_{\switchpar}(X^n)} 
  = 0
  \quad \quad \text{with $\mbayes{P}_{k^*}$-probability 1.}
\end{equation}
Note that the $q_{\switchpar}$ in (\ref{eqn:badcasesdieb}) are defined
relative to the non-Bayesian estimators $\m{p}_1, \m{p}_2, \ldots$,
whereas the $\mbayes{P}_{k^*}$ on the right of the equation is the
probability according to a \emph{Bayesian} marginal distribution
$\mbayes{P}_{k^*}$, which has been chosen so that the theorem's
condition holds. To see that (\ref{eqn:badcasesdieb}) is sufficient to
prove the theorem, suppose the theorem is false. Then, because the prior
$w_{k^*}$ is mutually absolutely continuous with Lebesgue measure, there
exists a set of parameters $\Phi \subseteq \Theta_{k^*}$ with nonzero
prior measure under $w_{k^*}$, such that \eqref{eqn:thmconsistencyb}
does not hold for any $\theta^* \in \Phi$. But then by definition of
$\mbayes{P}_{k^*}$ we have a contradiction with
\eqref{eqn:badcasesdieb}.

Using exactly the same reasoning as in the proof of
Theorem~\ref{thm:consistencyagain}, it follows that, analogously to
(\ref{eqn:Adies}), we have
\begin{equation}
  \label{eqn:Adiesb}
  \lim_{n \to \infty} \frac{r(X^n)}
    {\pi(\switchpar=(0,k^*))\cdot \mbayes{p}_{k^*}(X^n)} = 0
  \quad \quad \text{with $\mbayes{P}_{k^*}$-probability 1.}
\end{equation}
This is just (\ref{eqn:Adies}) with $r$ now referring to a mixture of
combinator prediction strategies defined relative to the non-Bayesian
estimators $\m{p}_1, \m{p}_2, \ldots$, and the $\mbayes{p}_{k^*}$ in the
denominator and on the right referring to the Bayesian marginal
distribution $\mbayes{P}_{k^*}$. Using \eqref{eqn:onlyAmatters} and the
fact that
 $ \sum_{\switchpar \in \switchpars} \pi(\switchpar) q_{\switchpar}(x^n)
    \geq \pi(\switchpar = (0,k^*))\cdot \pbayes_{k^*}(x^n)$,
    and the fact that, by assumption, for some $K$, for all large $n$,
    $\pbayes_{k^*}(X^n) \geq
    \mbayes{p}_{k^*}(X^n)2^{-K}$ with $\mbayes{P}_{k^*}$-probability 1,
    (\ref{eqn:Adiesb}) implies (\ref{eqn:badcasesdieb}).  \qed

\subsection{Mutual Singularity as Used in the Proof of Theorem~\ref{thm:consistencyagain}}
\label{sec:harremoes}
\newcommand{\algebra}{\ensuremath{\mathcal{A}}}
\newcommand{\borel}{\ensuremath{\mathcal{B}}}

Let $Y^2 = (Y_1$, $Y_2)$ be random variables that take values in
separable metric spaces $\Omega_1$ and $\Omega_2$, respectively. We
will assume all spaces to be equipped with Borel $\sigma$-algebras
generated by the open sets. Let $p$ be a prediction strategy for $Y^2$
with corresponding distributions $P(Y_1)$ and, for any $y^1 \in
\Omega_1$, $P(Y_2|y^1)$. To ensure that $P(Y^2)$ is well-defined, we
impose the requirement that for any fixed measurable event $A_2
\subseteq \Omega_2$ the probability $P(A_2 | y^1)$ is a measurable
function of $y^1$.

\begin{lemma}
  \label{lem:eventuallysingular}
  Suppose $p$ and $q$ are prediction strategies for $Y^2 = (Y_1, Y_2)$,
  which take values in separable metric spaces $\Omega_1$ and
  $\Omega_2$, respectively. Then if $P(Y_2|y^1)$ and $Q(Y_2|y^1)$ are
  mutually singular for all $y^1 \in \Omega_1$, then $P(Y^2)$ and
  $Q(Y^2)$ are mutually singular.
\end{lemma}

The proof, due to Peter Harremo\"es, is given below the following
corollary, which is what we are really interested in. Let $X^\infty =
X_1$, $X_2$, $\ldots$ be random variables that take values in the
separable metric space $\samplespace$. Then what we need in the proof
of Theorem~\ref{thm:consistencyagain} is the following corollary of
Lemma~\ref{lem:eventuallysingular}:
\begin{corollary}
\label{cor:harremoes}
  Suppose $p$ and $q$ are prediction strategies for the sequence of random variables
  $X^\infty = X_1$, $X_2$, $\ldots$ that take values in respective separable
  metric spaces $\samplespace_1$, $\samplespace_2$, $\ldots$ Let
  $m$ be any positive integer. Then if $P(X_{m+1}^\infty|x^m)$ and
  $Q(X_{m+1}^\infty|x^m)$ are mutually singular for all $x^m \in
  \samplespace^m$, then $P(X^\infty)$ and $Q(X^\infty)$ are mutually
  singular.
\end{corollary}

\begin{proof}
  The product spaces $\samplespace_1 \times \cdots \times
  \samplespace_m$ and $\samplespace_{m+1} \times \samplespace_{m+2}
  \times \cdots$ are separable metric spaces
  \cite[pp.\ 5,6]{parthasarathy1967}. Now apply
  Lemma~\ref{lem:eventuallysingular} with $\Omega_1 = \samplespace_1
  \times \cdots \times \samplespace_m$ and $\Omega_2 =
  \samplespace_{m+1} \times \samplespace_{m+2} \times \cdots$.
\end{proof}

\begin{proof}[Proof of Lemma~\ref{lem:eventuallysingular}]
  For each $\omega_1 \in \Omega_1$, by mutual singularity of
  $P(Y_2|\omega_1)$ and $Q(Y_2|\omega_1)$ there exists a
  measurable set $C_{\omega_1} \subseteq \Omega_2$ such that
  $P(C_{\omega_1}|\omega_1) = 1$ and $Q(C_{\omega_1}|\omega_1) = 0$. As
  $\Omega_2$ is a metric space, it follows from \cite[Theorems~1.1 and
  1.2 in Chapter~II]{parthasarathy1967} that for any $\epsilon > 0$
  there exists an open set $U_{\omega_1}^\epsilon \supseteq
  C_{\omega_1}$ such that
  \begin{equation}
    P(U_{\omega_1}^\epsilon|\omega_1) = 1
    \quad \text{and} \quad
    Q(U_{\omega_1}^\epsilon|\omega_1) < \epsilon.
  \end{equation}

  As $\Omega_2$ is a separable metric space, there also exists a
  countable sequence $\{B_i\}_{i\geq1}$ of open sets such that every
  open subset of $\Omega_2$ ($U_{\omega_1}^\epsilon$ in particular) can
  be expressed as the union of sets from $\{B_i\}$ \cite[Theorem~1.8 in
  Chapter I]{parthasarathy1967}.

  Let $\{B_i'\}_{i\geq1}$ denote a subsequence of $\{B_i\}$ such that
  $U_{\omega_1}^\epsilon = \Union_i B'_i$. Suppose $\{B_i'\}$ is a
  finite sequence. Then let $V_{\omega_1}^\epsilon =
  U_{\omega_1}^\epsilon$. Suppose it is not. Then $1 =
  P(U_{\omega_1}^\epsilon|\omega_1) = P(\Union_{i=1}^\infty
  B'_i|\omega_1) = \lim_{n \rightarrow \infty} P(\Union_{i=1}^n
  B'_i|\omega_1)$, because $\Union_{i=1}^n B'_i$ as a function of $n$ is
  an increasing sequence of sets. Consequently, there exists an $N$ such
  that $P(\Union_{i=1}^N B'_i|\omega_1) > 1 - \epsilon$ and we let
  $V_{\omega_1}^\epsilon = \Union_{i=1}^N B'_i$. Thus in any case there
  exists a set $V_{\omega_1}^\epsilon \subseteq U_{\omega_1}^\epsilon$
  that is a union of a finite number of elements in $\{B_i\}$ such that
  \begin{equation}
    P(V_{\omega_1}^\epsilon|\omega_1) > 1 - \epsilon
    \quad \text{and} \quad
    Q(V_{\omega_1}^\epsilon|\omega_1) < \epsilon.
  \end{equation}

  Let $\{D\}_{i\geq1}$ denote an enumeration of all possible unions of a
  finite number of elements in $\{B_i\}$ and define the disjoint
  sequence of sets $\{A^\epsilon_i\}_{i\geq1}$ by
  \begin{equation}
    A_i^\epsilon = \{\omega_1 \in \Omega_1 :
                      P(D_i|\omega_1) > 1 - \epsilon,
                      Q(D_i|\omega_1) < \epsilon\}
                  \setminus \Union_{j=1}^{i-1} A_j^\epsilon
  \end{equation}
  for $i = 1$, $2$, $\ldots$ Note that, by the reasoning above, for each
  $\omega_1 \in \Omega_1$ there exists an $i$ such that $\omega_1 \in
  A_i^\epsilon$, which implies that $\{A^\epsilon_i\}$ forms a partition
  of $\Omega_1$. Now, as all elements of $\{A^\epsilon_i\}$ and $\{D_i\}$
  are measurable, so is the set $F^\epsilon = \Union_{i=1}^\infty
  A_i^\epsilon \times D_i \subseteq \Omega_1 \times \Omega_2$, for which
  we have that $P(F^\epsilon) = \sum_{i=1}^\infty P(A_i^\epsilon \times
  D_i) > (1-\epsilon) \sum_{i=1}^\infty P(A_i) = 1-\epsilon$ and
  likewise $Q(F^\epsilon) < \epsilon$.

  Finally, let $G = \Intersection_{n=1}^\infty \Union_{k=n}^\infty
  F^{2^{-k}}$. Then $P(G) = \lim_{n \rightarrow \infty}
  P(\Union_{k=n}^\infty F^{2^{-k}}) \geq \lim_{n \rightarrow \infty} 1 -
  2^{-n} = 1$ and $Q(G) = \lim_{n \rightarrow \infty} Q(\Union_{k=n}
  F^{2^{-k}}) \leq \lim_{n \rightarrow \infty} \sum_{k=n}^\infty 2^{-k}
  = \lim_{n \rightarrow \infty} 2^{-n+1} = 0$, which proves the lemma.
\end{proof}
\subsection{Proofs of Section~\ref{sec:nonparametrica}}
\subsubsection{Proof of Lemma~\ref{lem:switchvsbma}}
  For the first part we underestimate sums:
  \begin{align*}
    \pswitch(x^n)
      &=\sum_{m\in\posints}\sum_{\switchpar\in\switchpars:m(\switchpar)=m}
          q_\switchpar(x^n)\switchprior(\switchpar)
      ~\geq~\pim(1)\cdot \sum_{k'\in\posints} \pik(k')\pbayes_{k'}(x^n)
      ~=~\pim(1)\cdot\pbma(x^n),\\
    \pbma(x^n)
      &=\sum_{k'\in\posints}\pbayes_{k'}(x^n)\pik(k')
      ~\geq~\pik(k)\pbayes_k(x^n).
  \end{align*}
  We apply~\eqref{eq:redundancyrisk} to bound the difference in
  cumulative risk from above:
  \begin{alignat*}{4}
    \sum_{i=1}^n \risk_i(P^*,\Pswitch) &=
    E\left[\log\frac{p^*(X^n)}{\pswitch(X^n)}\right]&&\le E\left[\log\frac{p^*(X^n)}{\pim(1)\pbma(X^n)}\right]&&=\sum_{i=1}^n \risk_i(P^*,\Pbma)&&-\log\pim(1),\\
    \sum_{i=1}^n \risk_i(P^*,\Pbma) &=
    E\left[\log\frac{p^*(X^n)}{\pbma(X^n)}\right]&&\le
    E\left[\log\frac{p^*(X^n)}{\pik(k)\pbayes_k(X^n)}\right]&&=\sum_{i=1}^n \risk_i(P^*,\Pbayes_k)&&-\log\pik(k). \ \  \ \mbox{$\Box$}
\qedhere
  \end{alignat*}
%
%

\subsubsection{Proof of Theorem~\ref{thm:switchrsy}} 
We will prove a
slightly stronger version of the theorem, which shows that the
switch distribution in fact achieves the same multiplicative constant,
$A$, as is shown in \citep{rissanen1992} for the estimator that selects
$\ceil{n^{1/3}}$ bins:
\begin{equation}
  \sup_{P^* \in \model^*} \sum_{i=1}^n \risk_i(P^*,\Pswitch)
    \preceq_1 A\, n^{1/3}.
\end{equation}

The idea of the proof is to exhibit an oracle that closely approximates
the estimator $\m{P}_{\ceil{n^{1/3}}}$, but only switches a logarithmic
number of times in $n$ on the first $n$ outcomes, and then apply
Lemma~\ref{lem:rateofconvergence} to this oracle.

In \citep{rissanen1992} Equation~\ref{eqn:rsy1} is proved from the
following theorem, which gives an upper bound on the risk of any
prediction strategy that uses a histogram model with approximately
$\ceil{n^{1/3}}$ bins to predict outcome $X_{n+1}$:
\begin{theorem}
  \label{thm:rsy2}
  For any $\alpha \geq 1$
  \begin{equation}
    \max_{\ceil{(n/\alpha)^{1/3}} \leq k \leq \ceil{n^{1/3}}}
      \sup_{P^* \in \model^*} \risk_n(P^*, \m{P}_k)
      \preceq_1 \alpha^{2/3} C n^{-2/3},
    \label{eqn:rsy2}
  \end{equation}
  where $C > 0$ depends only on $c_2$ in \eqref{eqn:rsydensities}.
\end{theorem}
In~\citep{rissanen1992} the theorem is only proved for $\alpha = 1$, but
their proof remains valid for any $\alpha > 1$. From this,
\eqref{eqn:rsy1} follows by summing \eqref{eqn:rsy2} and approximating
$\sum_{i=1}^n i^{-2/3}$ by an integral. Summation is allowed, because
$\risk_i(P^*,\m{P}_k)$ is finite for all $P^* \in \model^*$, $i$ and $k$,
and $\alpha^{2/3} C \sum_{i=1}^n i^{-2/3} \to \infty$ as $n$ goes to
infinity. The constant $A$ in \eqref{eqn:rsy1} is the product of $C$ and
the approximation error of this integral approximation. We will now
apply Theorem~\ref{thm:rsy2} to prove Theorem~\ref{thm:switchrsy} as
well.

Let $\alpha > 1$ be arbitrary and let $t_j = \ceil{\alpha^{j-1}}-1$ for
$j\in\posints$ be a sequence of switch-points. For any $n$, let $j_n$
denote the index of the last preceding switch-point, i.e.\
$n\in[t_j+1,t_{j+1}]$. Now define the oracle $\oracle_{\alpha(P^*,x^{n-1})}
:= \ceil{(t_{j_n}+1)^{1/3}}$ for any $P^* \in \model^*$ and any $x^{n-1}
\in \samplespace^{n-1}$. If we can apply
Lemma~\ref{lem:rateofconvergence} to $\oracle_\alpha$ with
$f(n)=n^{1/3}$, $g(n) = \ceil{n^{1/3}}$, $c_1 = \alpha^{2/3}A$ and $c_2
= 0$, we will obtain
\begin{equation}
  \limsup_{n \to \infty} \frac{\sum_{i=1}^n
    \risk_i(P^*,\Pswitch)}{n^{1/3}} \leq \alpha^{2/3} A
  \label{eqn:withalpha}
\end{equation}
for any $\alpha > 1$. Theorem~\ref{thm:switchrsy} then follows,
because the left-hand side of this expression does not depend on
$\alpha$. It remains to show that conditions (i)--(iii) of
Lemma~\ref{lem:rateofconvergence} are satisfied.

Condition~\ref{cond:canapplypreviouslemma} follows because $t_{j_n} + 1
\leq n$. Condition~\ref{cond:ii} is implied by the fact that
$\oracle_\alpha$ has only a logarithmic number of switch-points: It
satisfies $m_{\oracle_\alpha}(n) \leq \ceil{\log_\alpha n}+2$.
Consequently,
\begin{equation}
  m_{\oracle_\alpha}(n)(\log n + \log g(n))
    \leq (\ceil{\log_\alpha n}+2)(\log n + \ceil{n^{1/3}})
    = o(n^{1/3}).
\end{equation}

To verify Condition~\ref{cond:iii}, note that the selected number of
bins is close to $\ceil{n^{1/3}}$ in the sense of
Theorem~\ref{thm:rsy2}: For $n \in [t_j+1,t_{j+1}]$ it follows from
$(t_{j+1})/(t_j+1) \leq \alpha$ that
\begin{equation}
  \left\lceil(t_j+1)^{1/3} \right\rceil
  = \left\lceil \left(\frac{n}{n/(t_j+1)}\right)^{1/3} \right\rceil
  \in \left[\ceil{(n/\alpha)^{1/3}},\ceil{n^{1/3}}\right].
\end{equation}
We can therefore apply Theorem~\ref{thm:rsy2} to obtain
\begin{align}
  \sup_{P^* \in \model^*} \sum_{i=1}^n \risk_i(P^*,\m{P}_{\oracle_\alpha})
    \leq \sum_{i=1}^n \sup_{P^* \in \model^*} \risk_i(P^*,\m{P}_{\oracle_\alpha})
    \preceq_1 \alpha^{2/3} C \sum_{i=1}^n i^{-2/3}
    \preceq_1 \alpha^{2/3} A\, n^{1/3}.
  \label{eqn:oracleinsum}
\end{align}
This shows that Condition~\ref{cond:iii} is satisfied and
Lemma~\ref{lem:rateofconvergence} can be applied to prove the theorem.
\mbox{$\Box$}
\subsection{Proof of Proposition~\ref{prop:varfix} and Proposition~\ref{prop:yang}}
\label{sec:varfix}
We will actually prove a more general proposition that implies both
Proposition~\ref{prop:varfix} and~\ref{prop:yang}. 
Let $\Pmm$ be any prediction strategy. Now define the
prediction strategy $$ \Pcesaro(x_{n} \mid x^{n-1}) :=
\frac{1}{n} \sum_{i=1}^n \Pmm(x_n \mid x^{i-1}). $$ Thus, $\Pcesaro$
is obtained as a time (``Ces\`aro''-) average of $\Pmm$.
\begin{proposition}
\label{prop:cesaro}
Suppose that $\model^*$ is standard nonparametric, and that $\Pmm$
achieves the minimax rate in Ces\`aro mean, i.e.\  $\sup_{P^* \in
  \model^*} \sum_{i=1}^n \risk_i(P^*, \Pmm) \preceq \Gmmfix(n)$. Then
$$
\gmm(n) \leq \sup_{P^* \in \model^*} \risk_n(P^*,\Pcesaro)
\preceq n^{-1} \Gmmfix(n) \leq n^{-1} \Gmmvar(n) \preceq \gmm(n).
$$
\end{proposition}
\begin{proof}{\bf \ (of Proposition~\ref{prop:cesaro})}
 We show this by extending an argument from \cite[p.~1582]{YangB99}.
 By applying Jensen's inequality as in
 Proposition~15.2 of \citep{grunwald2007} (or the corresponding results
 in \citep{Yang00} or \citep{YangB99}) it now follows that, for all $P^*
 \in \model^*$, $\risk_n(P^*, \Pcesaro) \leq \frac{1}{n} \sum_{i=1}^n
 \risk_i(P^*, \Pmm)$, so that also
\begin{equation}
  \label{eq:cesaro}
  \sup_{P^*} \risk_n(P^*, \Pcesaro)
    \leq \sup_{P^*} \frac{1}{n} \sum_{i=1}^n \risk_i(P^*, \Pmm).
\end{equation}
This implies that
\begin{equation*}
n \gmm(n) \leq n \cdot \sup_{P^*} \risk_n(P^*, \Pcesaro)
    \preceq \Gmmfix(n)
    \leq \Gmmvar(n) = \sum_{i=1}^n \gmm(i).
\end{equation*}
Therefore, it suffices to show that for standard nonparametric models,
$ \sum_{i=1}^n \gmm(i) \preceq n \gmm(n)$. By (\ref{eq:hgrowth}),
$\gmm(i) \asymp i^{-\gamma} h_0(i)$ for some {\em increasing\/}
function $h_0$. Then
\begin{equation}
  \label{eq:sumtrick}
  \sum_{i=1}^n \gmm(i)
    = \sum_{i=1}^n i^{- \gamma} h_0(i) 
    \leq h_0(n) \sum_{i=1}^n i^{- \gamma}
    \overset{(a)}{\preceq} h_0(n) n^{1 - \gamma}
    = n \cdot n^{- \gamma} h_0(n)
    \asymp n \gmm(n).
\end{equation}
where (a) follows by approximating the sum by an integral. The result follows. 
\end{proof}
\subsection{Proof of Lemma~\ref{lem:linregb}}
\begin{proof}
 Let $P^* \in \model^*$ be arbitrary. We may transform $\phi_1$ to
 $\psi_1$, $\phi_2$ to $\psi_2$ and so on, such that for each $k$,
 $(\psi_1, \ldots, \psi_k)$ is an orthonormal basis for $S_k$ with
 respect to $P^*$. For any $k$, each $P \in \model_k$ may now be
 parameterized by $\eta = (\eta_{(1)}, \ldots, \eta_{(k)}) \in \reals^k$,
 which means that $P_\eta \equiv P$ expresses $Y_i = \sum_{j=1}^k \eta_{(j)}
 \psi_j(X_i) + U_i$. Now let $k \in \posints$ be arbitrary and define
 $\tilde{\eta}$ such that $\tilde{P}_k = P_{\tilde{\eta}}$. Let $\psi
 := ( \psi_1, \ldots, \psi_k)^\transpose$.  Using the fact that the
 errors are normally distributed, for any $\eta \in \reals^k$,
 abbreviating $\psi(X)$ to $\psi$, we have
\begin{eqnarray}
\lefteqn{D(P^* \| P_{\eta}) - D(P^* \| P_{\tilde{\eta}}) 
=} & & \nonumber \\
& = & \frac{1}{2 \sigma^2} E E \left[ 
(Y - \eta^\transpose \psi)^2 - (Y - \tilde{\eta}^\transpose \psi)^2
 \mid X\right] \nonumber \\
& = & \frac{1}{2 \sigma^2} E \left[ 
- 2 E[Y|X] (\eta - \tilde{\eta})^\transpose
\psi + (\eta^\transpose \psi)^2 
- (\tilde{\eta}^\transpose \psi)^2 \right] \nonumber \\
& = & \frac{1}{2 \sigma^2} E \left[ 
- 2 \left(\sum_{j=1}^k \tilde{\eta}_{(j)} \psi 
+ \sum_{j={k+1}}^\infty \tilde{\eta}_{(j)} \psi
\right) 
\left(\sum_{j=1}^k (\eta_{(j)} - \tilde{\eta}_{(j}) \psi
\right) + (\eta^\transpose \psi)^2 
- (\tilde{\eta}^\transpose \psi)^2 \right] \nonumber \\
& = & \frac{1}{2 \sigma^2} E \left[ 
- 2 \sum_{j=1}^k \tilde{\eta}_{(j)} \psi \left(\sum_{j=1}^k (\eta_j - \tilde{\eta}_{(j)}) \psi
\right)
-2 B + (\eta^\transpose \psi)^2 
- (\tilde{\eta}^\transpose \psi)^2 \right] \nonumber \\
& = & 
\frac{1}{2 \sigma^2} E \left[ 
- 2 \left(\sum_{j=1}^k \tilde{\eta}_{(j)} \psi \right) 
\left(\sum_{j=1}^k (\eta_j - \tilde{\eta}_{(j)}) \psi 
\right) + (\eta^\transpose \psi)^2 
- (\tilde{\eta}^\transpose \psi)^2 \right] \nonumber \\
& = & \frac{1}{2 \sigma^2} E\left[
(\tilde{\eta}^\transpose \psi)^2 
-2 \left(\tilde{\eta}^\transpose  \psi\right)
\left(\eta^\transpose
\psi \right)
+ (\eta^\transpose \psi)^2 \right] \nonumber \\
& = &
\frac{1}{2 \sigma^2}
E\left[ (\tilde{\eta}^\transpose \Psi - \eta^\transpose \Psi)^2
\right] = 
\frac{1}{2 \sigma^2}
E \left[(\tilde{\eta} - \eta)^\transpose \psi
\psi^\transpose (\tilde{\eta} - \eta) \right] \notag\\
&= &
\frac{1}{2 \sigma^2}
(\tilde{\eta} - \eta)^\transpose (\tilde{\eta} - \eta).\label{eq:lenzen}
\end{eqnarray}
Here the outer expectation on each line is expectation according to $P^*_X$,
the marginal distribution of $X$ under $P^*$. In the fourth equality, 
$B = \left(\sum_{j={k+1}}^\infty \tilde{\eta}_{(j)} \psi\right)
\left(\sum_{j=1}^k (\eta_{(j)} - \tilde{\eta}_{(j)}) \psi
\right)$, which,  by orthogonality of the $\psi_j$, is equal to $0$. The final
equality also follows by orthogonality. 

Now fix $n> k$, and let $\hat{\eta}_{n}$ denote the maximum likelihood
parameter value in the $\eta$-parameterization based on data
$X^{n-1}$, i.e.\ $P_{\hat{\eta}_{n}} := \m{P}_k(Y_n = \cdot \mid X^{n}, Y^{n-1})$
(note that $\m{P}_k(Y_n = \cdot \mid X^{n}, Y^{n-1})$ itself does not depend on
the choice of basis).  Using (\ref{eq:lenzen}), we can rewrite
(\ref{eq:estimerror}) as follows:
\begin{equation}
\label{eq:piep} 
E \left
[(\tilde{\eta}_{n-1} - \hat{\eta}_{n-1})^\transpose (\tilde{\eta}_{n-1} -
\hat{\eta}_{n-1})
\right] 
\geq 
E \left
[(\tilde{\eta}_{n} - \hat{\eta}_{n})^\transpose (\tilde{\eta}_{n} -
\hat{\eta}_{n})
\right], 
\end{equation}
where now the expectation is over $X^{n-1}, Y^{n-1}$, sampled i.i.d.
from $P^*$.  It thus remains to show that (\ref{eq:piep}) holds.

Write $\Psi^{(n)}$ for the $n \times k$ design matrix with
$(j,i)$-th entry given by $\psi_j(x_i)$.  We show
further below that, if $x_1, \ldots, x_{n-1}$ are such that $(\Psi^{(n-1)})^\transpose
\Psi^{(n-1)}$ is nonsingular, then the variance of $\hat{\eta}_{n-1}$ is at least as large as the variance of $\hat{\eta}_{n}$, i.e.:
\begin{equation}
\label{eq:siep}
E   [(\tilde{\eta} - \hat{\eta}_{n-1})^\transpose 
(\tilde{\eta} - \hat{\eta}_{n-1}) \mid X^{n} = x^{n}]
\geq 
E [ (\tilde{\eta} - \hat{\eta}_{n})^\transpose 
(\tilde{\eta} - \hat{\eta}_{n}) \mid X^{n} = x^{n}].
\end{equation}
Since, by our assumptions. for all $k$, all $n$, 
$$E[ (\tilde{\eta} - \hat{\eta}_{n})^\transpose 
(\tilde{\eta} - \hat{\eta}_{n}) \mid
\text{$(\Psi^{(n)})^\transpose \Psi^{(n)}$ is singular}  ] < \infty,
$$
where, also by assumption, the event that $(\Psi^{n)})^\transpose \Psi^{(n)}$ is singular
has $P^*$-measure 0, it follows that (\ref{eq:piep}) is
implied by (\ref{eq:siep}). Thus, it remains to prove (\ref{eq:siep}).  We prove (\ref{eq:siep}) by slightly adjusting an existing geometric 
proof of the related (but non-equivalent) Gauss-Markov theorem \citep{Ruud95}.
Define, for given $x^n$, 
$$
P  =  \Psi^{(n)} \left( \left( \Psi^{(n)}
  \right)^\transpose \Psi^{(n)} \right)^{-1} \left(\Psi^{(n)}
\right)^\transpose \ \ ; \ \ 
Q = \Psi^{(n)} \left( \left( \Psi^{(n-1)}
  \right)^\transpose \Psi^{(n-1)} \right)^{-1} \left(\Psi^{(n-1)} \right)^\transpose J,
$$ 
where $J$ is the $(n-1) \times n$ matrix with $J_{1,1} = \ldots =
  J_{n-1,n-1} = 1$, and all other entries equal to $0$. Letting ${\bf
    y} = (y_1, \ldots, y_n)^\transpose$, we see that $P$ is a
  projection matrix, and
\begin{equation}
\label{eq:bieb}
P {\bf y} = \Psi^{(n)} \hat{\eta}_{n}
\ \ ; \ \ 
Q {\bf y} = 
\Psi^{(n)} \hat{\eta}_{n-1}. 
\end{equation}
Now, for arbitrary $a \in \reals^n$, we have
\begin{equation}
\label{eq:geintje}
\var(a^\transpose Q {\bf y} \mid x^n) =
\var(a^\transpose P {\bf y} \mid x^n) + \var(a^\transpose (Q-P) {\bf y} \mid x^n)
+ 2 \cov (a^\transpose(Q-P) {\bf y}, a^\transpose P {\bf y} \mid x^n).
\end{equation}
A straightforward (but tedious) calculation shows that 
$$
\cov (a^\transpose(Q-P) {\bf y}, a^\transpose P {\bf y} \mid x^n)
= \sigma^2 a^\transpose (QP^\transpose - P P^\transpose) a.
$$ 
As $P$ is symmetric, $P^\transpose = P$, and for all ${\bf y} \in \reals^n$,
$\bar{\bf y} := P {\bf y}$ is in the column space of $\Psi^{(n)}$,
so that $P \bar{\bf y} = \bar{\bf y}$, and $P P^\transpose {\bf y} =
\bar{\bf y}$. But since $Q \Psi^{(n)} = \Psi^{(n)}$, and $\bar{\bf y}$ is in the column space of $\Psi^{(n)}$, we must also have $Q
\bar{\bf y} = \bar{\bf y}$ and $QP^\transpose {\bf y} = \bar{\bf
  y}$. Thus, for arbitrary ${\bf y}$, $QP^\transpose {\bf y} = P
P^\transpose {\bf y}$, and it follows that the $\cov$-term in
(\ref{eq:geintje}) is equal to $0$. Thus, (\ref{eq:geintje}) implies
that
\begin{equation}
\label{eq:heintje}
\var(a^\transpose Q {\bf y} \mid x^n) \geq
\var(a^\transpose P {\bf y} \mid x^n)
\end{equation}
Now apply this with $$a := (1,1, \ldots, 1)^\transpose \cdot 
\left( \left( \Psi^{(n)}
  \right)^\transpose \Psi^{(n)} \right)^{-1} \left(\Psi^{(n)}
\right)^\transpose
,$$ 
where the leftmost vector is a  $k$-dimensional vector of 1s.  
By (\ref{eq:bieb}), (\ref{eq:heintje}) now becomes equivalent to 
$\var \sum_{j=1}^{k}
\hat{\eta}_{n-1,j} \geq \var \sum_{j=1}^{k} \hat{\eta}_{n,j} $, which is just (\ref{eq:siep}).
\end{proof}
\subsection{Proof of Theorem~\ref{thm:algo}}\label{sec:proofalgo}
\newcommand{\expert}{\xi} 

Before we prove Theorem~\ref{thm:algo}, we need to establish some additional
properties of the prior $\pi$ as defined in~\eqref{eq:prior}. 
To this end, let us define the random variables
\begin{eqnarray}
S_n(\switchpar)& := &\ind_{(n-1) \in \{t_1,\ldots,t_m\}}(\switchpar),\\
M_n(\switchpar)& := &\ind_{n > t_m}(\switchpar)
\end{eqnarray}
for all $n\in\posints$ and $\switchpar=((t_1,k_1)$, $\ldots$,
$(t_m,k_m))\in\switchpars$. These functions denote, respectively,
whether or not a switch occurs between outcome $X_{n-1}$ and outcome
$X_n$, and whether or not the last switch occurs somewhere before
outcome $n$. We also define $\expert_n(\switchpar):=(S_n(\switchpar)$,
$M_n(\switchpar)$, $K_n(\switchpar))$ as a convenient abbreviation.

Every parameter value $\switchpar\in\switchpars$ determines an infinite
sequence of values $\expert_1$, $\expert_2$, $\ldots$, and vice versa.
The advantage of these new variables is that they allow us to interpret
the prior as a sequential strategy for prediction of the value of the
next random variable $\expert_{n+1}$ (which in turn determines the
distribution on $X_{n+1}$ given $x^n$), given all previous random
variables $\expert^n := (\expert_1, \ldots, \expert_n)$. In fact, we
will show that $\pswitch(\expert_{n+1} \mid X^n,\expert^n) =
\pi(\expert_{n+1} \mid \expert^n)$. We therefore first calculate the
conditional probability $\pi(\expert_{n+1}|\expert^n)$ before proceeding
to prove the theorem. As it turns out, our prior has the nice property
that $\pi(\expert_{n+1}\mid \expert^n) = \pi(\expert_{n+1} \mid M_n,
K_n)$, which is the reason for the efficiency of the algorithm.

\begin{lemma}\label{lemma:preqprior}
  Let $\pi(\switchpar)=\theta^{m-1}(1-\theta)\pik(k_1)\prod_{i=2}^m
  \pit(t_i|t_i>t_{i-1})\pik(k)$ as in~\eqref{eq:prior}. Then
  \begin{equation}
    \pi(\expert_1) =
      \begin{cases}
        \pik(K_1)\theta &\textnormal{if $M_1=0$},\\
        \pik(K_1)(1-\theta) &\textnormal{if $M_1=1$}.
      \end{cases}
      \label{eq:priore1}
  \end{equation}
   And for $n \geq 1$

  \begin{equation}
    \pi(\expert_{n+1}\mid\expert^n)
      = \pi(\expert_{n+1}\mid M_n,K_n), \text{and}
  \end{equation}
  \begin{align}
    \pi(\expert_{n+1}&=(s_{n+1},m_{n+1},k_{n+1})\mid M_n=m_n,K_n=k_n)\\
    &=
      \begin{cases}%
        \pit(T>n|T\ge n)&\textnormal{if $s_{n+1}=0$, $m_{n+1}=m_n=0$,
        $k_{n+1} = k_n$,}\\
        1&\textnormal{if $s_{n+1}=0$, $m_{n+1}=m_n=1$, $k_{n+1}=k_n$,}\\
        \pit(T=n|T\ge n)\pik(k_{n+1})\theta&\textnormal{if $s_{n+1}=1$, $m_{n+1}=m_n=0$,}\\
        \pit(T=n|T\ge n)\pik(k_{n+1})(1-\theta)&\textnormal{if $s_{n+1}=1$, $m_{n+1}=1$, $m_n=0$,}\\
        0 &\text{otherwise.}
      \end{cases}
      \label{eq:condprior}
  \end{align}
\end{lemma}
\begin{proof}
  To check \eqref{eq:priore1}, note that we must have either
  $\expert_1=(1,1,k)$ for some $k \in \posints$, which corresponds to
  $\switchpar=((0,k))$ which has probability $\pik(k)(1-\theta)$ as
  required, or $\expert_1=(1,0,k)$. The latter corresponds to the event
  that $m>1$ and $K_1=k$, which has probability $\pik(k)\theta$.

  We proceed to calculate the conditional probability
  $\pi(\expert_{n+1}\mid \expert^n)$ for $n \geq 1$. First suppose
  $M_n(\switchpar)=0$. Let $A_n(\switchpar):=\max\{i\mid
  t_i<n\}=\sum_{i=1}^n S_i$ count the number of switches before $n$.
  Also note that $\expert^n$ and $M_n=0$ determine $t_1$, $\ldots$,
  $t_{A_n}$, $k_1$, $\ldots$, $k_{A_n}$, that $t_{A_n} \geq n$ and
  $m(\switchpar) > A_n$, and vice versa. Hence for any $n$
  \begin{multline}
    \pi(\expert^n \text{ such that } M_n=0) = \\
     \pim(m > A_n)\, \pi(t_1,\ldots,t_{A_n},n \leq t_{A_n+1},k_1,\ldots,k_{A_n} \mid t_1< \ldots < t_{A_n + 1}, m > A_n).
    \label{eq:mneq0}
  \end{multline}
  Likewise, for $M_n = 1$
  \begin{equation}
  \pi(\expert^n \text{ such that } M_n = 1) =
     \pim(m = A_n)\, \pi(t_1,\ldots,t_{A_n},k_1,\ldots,k_{A_n} \mid t_1< \ldots < t_{A_n}, m = A_n).
     \label{eq:mneq1}
  \end{equation}
  From \eqref{eq:mneq0} and~\eqref{eq:mneq1} we can compute the
  conditional probability $\pi(\expert_{n+1}\mid \expert^n)$. We
  distinguish further on the basis of the possible values of $S_{n+1}$
  and $M_{n+1}$. Note that $M_{n+1}=0$ implies $M_n=0$ and $M_{n+1}=1$
  implies $M_n=1-S_{n+1}$. Also note that $S_{n+1}=0$ implies
  $A_{n+1}=A_n$ and $K_{n} = K_{n+1}$, and that $S_{n+1}=1$ implies
  $A_{n+1}=A_n+1$ and $t_{A_n+1}=n$. Conveniently, most factors cancel
  out, and we obtain
  \begin{eqnarray}
    \pi(S_{n+1} = 0, M_{n+1} = 0, K_{n+1} = k 
            \mid \expert^n \text{ s.t. } M_n = 0, K_n = k)
      &=& \pi(t_{A_n + 1} \geq n+1 \mid t_{A_n + 1} \geq n) \notag\\
      &=& \pit(T > n \mid T \geq n),\\
    \pi(S_{n+1} = 0, M_{n+1} = 1, K_{n+1} = k 
            \mid \expert^n \text{ s.t. } M_n = 1, K_n = k)
      &=& 1,
  \end{eqnarray}
  \begin{align}
    \pi(S_{n+1} = 1, M_{n+1} = 0, K_{n+1} = k 
            &\mid \expert^n \text{ s.t. } M_n = 0) \notag\\
      &= \pim(m > A_n + 1 \mid m > A_n)
          \pi(t_{A_n +1} = n \mid t_{A_n +1} \geq n)
          \pik(k) \notag\\
      &= \theta \, \pit(T = n \mid T \geq n) \pik(k),\\
    \pi(S_{n+1} = 1, M_{n+1} = 1, K_{n+1} = k 
            &\mid \expert^n \text{ s.t. } M_n = 0) \notag\\
      &= \pim(m = A_n + 1 \mid m > A_n)
          \pi(t_{A_n +1} = n \mid t_{A_n +1} \geq n)
          \pik(k) \notag\\
      &= (1-\theta) \, \pit(T = n \mid T \geq n) \pik(k).
  \end{align}
  The observation that these conditional probabilities depend only on
  $M_n$ and $K_n$ shows that $\pi(\expert_{n+1} \mid \expert^n) =
  \pi(\expert_{n+1} \mid M_n,K_n)$, which completes the proof of the
  lemma.
\end{proof}
\begin{proof}[Proof of Theorem~\ref{thm:algo}]
  Note that $\expert^n(\switchpar)$ completely determines
  $q_\switchpar(X^n)$. Therefore let $q_{\expert^n}(X^n) \equiv
  q_{\expert^n(\switchpar)}(X^n)$ $\equiv$ $q_{\switchpar}(X^n)$. It follows
  that
  \begin{eqnarray}
    \pswitch(\expert^{n+1}=e^{n+1},X^n)
      &=& \sum_{\switchpar : \expert^{n+1}=e^{n+1}} \pi(\switchpar) q_\switchpar(X^n) \\
      &=& q_{\expert^n}(X^n) \sum_{\switchpar : \expert^{n+1}(\switchpar)=e^{n+1}} \pi(\switchpar) \\
      &=& q_{\expert^n}(X^n) \pi(\expert^n = e^n) \pi(\expert_{n+1} = e_{n+1} \mid \expert^n = e^n) \\
      &=& \pswitch(\expert^n = e^n,X^n) \pi(\expert_{n+1} = e_{n+1} \mid \expert^n = e^n),
  \end{eqnarray}
  which together with Lemma~\ref{lemma:preqprior} implies that
  \begin{equation}
    \pswitch(\expert_{n+1} \mid \expert^n,X^n)
      = \pi(\expert_{n+1} \mid \expert^n)
      = \pi(\expert_{n+1} \mid M_n,K_n).
      \label{eq:indepdata}
  \end{equation}

  We will now go through the algorithm step by step to show that the
  invariants $w^a_k=P(x^{n-1},M_n=0,K_n=k)$ and
  $w^b_k=P(x^{n-1},M_n=1,K_n=k)$ hold for all $k \in \allowedk_n$ at the
  start of each iteration through the loop (before
  line~\ref{line:posterior}). These invariants ensure that
  $w^a_k+w^b_k=P(x^{n-1},K_n=k)$ so that the correct probabilities are
  reported.

  Line~\ref{line:init} initializes $w^a_k$ to $\pik(k)\theta =
  \pi(S_1=1,M_1=0,K_1=k) = \pswitch(x^0,M_1=0,K_1=k)$ for
  $k\in\allowedk_1$. Likewise $w^b_k = \pi(k)(1-\theta) =
  \pi(S_1=1,M_1=1,K_1=k) = \pswitch(x^0,M_1=1,K_1=k)$. Thus the loop
  invariant holds at the start of the first iteration.

  We proceed to show that the invariant holds in subsequent iterations
  as well. In the loss update in line~\ref{line:loss_update} we update
  the weights for $k\in\allowedk_n$ to
  \begin{eqnarray*}
    w^a_k & = &\pswitch(x^{n-1},M_n=0,K_n=k)\cdot p_k(x_n\mid x^{n-1})\\
    & =
    &\kern-1.5em\sum_{\switchpar:M_n=0,K_n=k}\kern-1.5em\pi(\switchpar)\left(\prod_{i=1}^{n-1}p_{K_i}(x_i\mid x^{i-1})\right)p_{K_n}(x_n\mid x^{n-1})\quad=\quad
    \pswitch(x^n,M_n=0,K_n=k).
\end{eqnarray*}
Similarly $w^b_k=\pswitch(x^n,M_n=1,K_n=k)$. Then in line~\ref{line:pool},
we compute $\texttt{pool}=\pit(Z=n\mid Z\ge
n)\sum_{k\in\allowedk_n}\pswitch(x^n,M_n=0,K_n=k)=\pit(Z=n\mid Z\ge n)\pswitch(x^n,M_n=0)$.
Finally, we consider the loop that starts at
line~\ref{line:share_update_start} and ends at
line~\ref{line:share_update_end}. First note that for $k \in
\allowedk_n$ by applying Lemma~\ref{lemma:preqprior} and
\eqref{eq:indepdata} we obtain
\begin{eqnarray}
\lefteqn{w^a_k\pit(Z>n\mid Z\ge n)
    =} & & \nonumber \\
&=& \pswitch(x^n,M_n=0,K_n=k)\pit(Z>n\mid Z\ge n)\nonumber\\
    &=&\pswitch(x^n,M_n=0,K_n=k)\pswitch(S_{n+1}=0,M_{n+1}=0,K_{n+1}=k\mid x^n,M_n=0,K_n=k) \nonumber\\
    &=&\pswitch(x^n,M_n=0,K_{n+1}=K_n=k,S_{n+1}=0,M_{n+1}=0) \nonumber\\
    &=&\pswitch(x^n,S_{n+1}=0,M_{n+1}=0,K_{n+1}=k).
  \label{eq:loop1}
\end{eqnarray}
Similarly we get for $k \in \allowedk_n$ that
\begin{equation}
  \label{eq:loop2}
    w^b_k = \pswitch(x^n,M_n=1,K_n=k)
    = \pswitch(x^n,S_{n+1}=0,M_{n+1}=1,K_{n+1}=k).
\end{equation}
As $S_{n+1} = 0$ implies $K_{n+1} = K_n$, we have for $k \in
\allowedk_{n+1}\setminus \allowedk_n$ that
\begin{eqnarray}
  \label{eq:loop3}
  \pswitch(x^n,S_{n+1}=0,M_{n+1}=0,K_{n+1}=k) &=& 0,\\
  \pswitch(x^n,S_{n+1}=0,M_{n+1}=0,K_{n+1}=k) &=& 0.
\end{eqnarray}
By Lemma~\ref{lemma:preqprior} and \eqref{eq:indepdata} we also get that
\begin{align}
  \texttt{pool}\pik(k)\theta
  &=\pit(Z=n\mid Z\ge n)\pswitch(x^n,M_n=0)\pik(k)\theta \notag\\
  &=\pswitch(x^n,M_n=0)\pswitch(S_{n+1}=1,M_{n+1}=0,K_{n+1}=k\mid x^n,M_n=0) \notag\\
  &=\pswitch(x^n,M_n=0,S_{n+1}=1,M_{n+1}=0,K_{n+1}=k) \notag\\
  &=\pswitch(x^n,S_{n+1}=1,M_{n+1}=0,K_{n+1}=k),
  \label{eq:loop4}
\end{align}
and similarly
\begin{multline}
  \label{eq:loop5}
  \texttt{pool}\pik(k)(1-\theta) =
  \pit(Z=n\mid Z\ge n)\pswitch(x^n,M_n=0)\pik(k)(1-\theta) \\
  =\pswitch(x^n,S_{n+1}=1,M_{n+1}=1,K_{n+1}=k).
\end{multline}
Together,
\eqref{eq:loop1},\eqref{eq:loop2},\eqref{eq:loop3},\eqref{eq:loop4}, and
\eqref{eq:loop5} imply that at the end of the loop
\begin{eqnarray*}
  w^a_k &=& \pswitch(x^n,S_{n+1}=0,M_{n+1}=0,K_{n+1}=k)
              + \pswitch(x^n,S_{n+1}=1,M_{n+1}=0,K_{n+1}=k)\\
        &=& \pswitch(x^n,M_{n+1}=0,K_{n+1}=k),\\
  w^b_k &=& \pswitch(x^n,S_{n+1}=0,M_{n+1}=1,K_{n+1}=k)
              +\pswitch(x^n,S_{n+1}=1,M_{n+1}=1,K_{n+1}=k) \\
        &=& \pswitch(x^n,M_{n+1}=1,K_{n+1}=k),
\end{eqnarray*}
which shows that the loop invariants hold at the start of the next
iteration and that after the last iteration the final posterior is also
correctly reported based on these weights.
\end{proof}
\bibliography{switch} \pagebreak
\tableofcontents
\end{document}